\documentclass[10pt]{amsart}
\usepackage[leqno]{amsmath}
\usepackage{amsthm}
\usepackage{amsfonts, hyperref, tikz-cd, enumerate}
\usepackage{amssymb}
\usepackage{eucal}
\usepackage[all]{xy}
\usepackage{mathtools}

\DeclareFontFamily{OT1}{rsfs}{}
\DeclareFontShape{OT1}{rsfs}{n}{it}{<-> rsfs10}{}
\DeclareMathAlphabet{\mathscr}{OT1}{rsfs}{n}{it}

\usepackage{xspace}
\usepackage{epsfig,epic,eepic,latexsym,color}
\usepackage[all]{xy}
\usepackage{comment} 
\numberwithin{equation}{section}

\usepackage{latexsym}

\theoremstyle{theorem} 

\newtheorem{theorem}{Theorem}[section]
\newtheorem{claim}[theorem]{Claim}
\newtheorem{proposition}[theorem]{Proposition}
\newtheorem{lemma}[theorem]{Lemma}

\newtheorem{corollary}[theorem]{Corollary}

\theoremstyle{definition}
\newtheorem{definition}[theorem]{Definition}
\newtheorem{remark}[theorem]{Remark}
\newtheorem{example}[theorem]{Example}

\newcommand{\zar}{\operatorname{Zar}}
\newcommand{\et}{\operatorname{\acute{e}t}}
\newcommand{\liset}{\operatorname{lis-\acute{e}t}}
\newcommand{\spec}{\operatorname{Spec}}

\newcommand{\we}{\widehat{\epsilon}}
\newcommand{\hd}{\widehat{\delta}}
\newcommand{\xhzar}{\widehat{X}_{\zar}}
\newcommand{\qc}{\operatorname{qc}}
\newcommand{\qcoh}{\operatorname{qcoh}}
\newcommand{\Hom}{\operatorname{Hom}}
\newcommand{\coh}{\operatorname{coh}}
\newcommand{\dm}{D(\operatorname{Mod}(\mathscr{O}_X))}
\newcommand{\dpm}{D(\operatorname{PMod}(\mathscr{O}_X))}
\newcommand{\fgt}{\operatorname{Fgt}}
\newcommand{\uhom}{\underline{\operatorname{Hom}}}
\newcommand{\uext}{\underline{\operatorname{Ext}}}

\title{Grothendieck's Existence Theorem for Relatively Perfect Complexes on Algebraic Stacks}
\author{David Benjamin Lim}

\address{Stanford University, Department of Mathematics, 450 Serra Mall
CA 94305, USA}
\email{benjamin.colim@gmail.com}
\date{\today}

\begin{document}
\maketitle

\begin{abstract}
We prove Grothendieck's existence theorem for relatively perfect complexes on an algebraic stack that is proper and flat over an $I$-adically complete Noetherian ring $A$. This generalizes an earlier result of Lieblich in the setting of algebraic spaces.
\end{abstract}

\setcounter{tocdepth}{1}
\tableofcontents

\section{Introduction}

Recently, pioneering work of Bridgeland \cite{ref:bridgeland} has shown that it is possible to define a notion of stability for objects in any triangulated category, vastly generalising the notion of stability of vector bundles as considered in GIT.
However, unlike moduli of vector bundles in GIT, it is a subtler question to construct  a moduli stack of objects in $D^b(\operatorname{Coh}(X))$, the bounded derived category of coherent sheaves on a scheme $X$. At the very least,  such a stack should be  \textit{algebraic} in order to ask geometric questions.
\medskip

  In \cite{ref:artinvers}, building on the seminal papers \cite{ref:artin1} and \cite{ref:artin2}, Artin gave a list of sufficient axioms for a stack to be algebraic. One of these axioms is that formal deformations of objects can be algebraized, i.e., that Grothendieck's existence theorem is satisfied. Therefore, in order to construct an algebraic stack parametrizing objects in $D^b(\operatorname{Coh}(X))$, we must at least be able to prove Grothendieck's existence theorem for such objects. The first result of this nature is that of Lieblich \cite[Proposition 3.6.1]{ref:lieblich}. He proves that if one restricts to relatively perfect complexes  (see Definition \ref{def:rp}) on an algebraic space $X$, proper and flat over a complete local Noetherian ring, then Grothendieck's existence theorem holds.  In fact, he checks the rest of Artin's axioms and proves that if $X \to S$ is a proper flat morphism of algebraic spaces,  the moduli space of relatively perfect and universally gluable complexes $\mathscr{D}_{X/S}$ is represented by an algebraic stack, locally of finite type over $S$ \cite[Theorem 4.2.1]{ref:lieblich}.
\medskip

The goal of this article is to establish an important first step towards constructing a moduli of relatively perfect objects on an \textit{algebraic stack}. More precisely, we  extend Lieblich's result on Grothendieck's existence theorem for relatively perfect complexes to the setting of algebraic stacks. This will be done by way of two more general results, namely formal GAGA and Grothendieck's existence theorem for pseudo-coherent complexes (c.f. Definition  \ref{def:pc} and Remark \ref{rem:pc}).

\begin{theorem}[Formal GAGA for pseudo-coherent complexes]
Let $\mathscr{X}$ be an algebraic stack that is proper over an $I$-adically complete Noetherian ring $A$. Then there is a natural equivalence of triangulated categories
\[ D^-_{\coh}(\mathscr{X}) \stackrel{\sim}{\to} D^-_{\coh}(\widehat{\hspace{0ex} \mathscr{X}})\]
of pseudo-coherent complexes.
\end{theorem}

\begin{theorem}[Grothendieck existence for pseudo-coherent complexes]
Let $\mathscr{X}$ be an algebraic stack that is proper over an $I$-adically complete Noetherian ring $A$. Define $\mathscr{X}_n \coloneqq \mathscr{X} \times_A A/I^{n+1}$.  Then any adic system of pseudo-coherent complexes on $\mathscr{X}_n$ algebraizes to a pseudo-coherent complex on $\mathscr{X}$.
\end{theorem}

\noindent Using these results, we can prove formal GAGA and Grothendieck's existence theorem for perfect complexes on algebraic stacks:
\begin{theorem}[Formal GAGA for perfect complexes]
Let $\mathscr{X}$ be an algebraic stack that is proper over an $I$-adically complete Noetherian ring $A$. Then there is a natural equivalence of triangulated categories
\[ \operatorname{Perf}(\mathscr{X}) \stackrel{\sim}{\to} \operatorname{Perf}(\widehat{\hspace{0ex} \mathscr{X}}) \]
of perfect complexes.
\end{theorem}

\begin{theorem}[Grothendieck existence for perfect complexes]
Let $\mathscr{X}$ be an algebraic stack that is proper  over an $I$-adically complete Noetherian ring $A$. Define $\mathscr{X}_n \coloneqq \mathscr{X} \times_A A/I^{n+1}$.  Then any adic system of perfect complexes on $\mathscr{X}_n$ algebraizes to a perfect complex on $\mathscr{X}$.
\end{theorem}

\noindent Finally, using Grothendieck's existence theorem for pseudo-coherent complexes, we deduce a generalization of \cite[Proposition 3.6.1]{ref:lieblich} to the setting of algebraic stacks. Note that unlike the case of pseudo-coherent or perfect complexes, we need a flatness hypothesis here.
\begin{theorem}[Grothendieck existence for relatively perfect complexes]
Let $\mathscr{X}$ be an algebraic stack that is proper  and flat over an $I$-adically complete Noetherian ring $A$. Define $\mathscr{X}_n \coloneqq \mathscr{X} \times_A A/I^{n+1}$.  Then any adic system of relatively perfect complexes on $\mathscr{X}_n$ algebraizes to a relatively perfect complex on $\mathscr{X}$.
\end{theorem}

\begin{remark}The reader may wonder if the triangulated category $D^-_{\coh}(\widehat{\hspace{0ex} \mathscr{X}})$ can be described as some kind of inverse limit of triangulated categories $\lim D^-_{\coh}(\mathscr{X}_n)$. At present, we do not know if there is a appropriate definition for such an inverse limit that does not use the language of $\infty$-categories. One issue is the definition of what morphisms should be in this category. For example,  given two systems $\{P_n\}$ and $\{Q_n\}$, defining  $\operatorname{Hom}(\{P_n\} , \{Q_n\} ) \coloneqq \lim \operatorname{Hom}(P_n,Q_n)$ is \textit{not} the right definition. This is because if we suppose that $P_n, Q_n$ are the reductions of $P,Q$ to $\mathscr{X}_n$, the natural surjective map
\[ \operatorname{Hom}(P,Q) \to \lim \operatorname{Hom}(P_n, Q_n) \]
has kernel $R^1\lim \operatorname{Hom}(P_n, Q_n[-1])$ which may not vanish. Since it is not necessary to deal with inverse limits of triangulated categories to state our theorems,  we will avoid the issue completely. \end{remark}

\medskip

\subsection{Comparison of Lieblich's result with ours} 

The classical form of Grothendieck's existence theorem asserts that given a scheme $X$ proper over an $I$-adically complete Noetherian ring $A$, any adic system of coherent sheaves $F_n$ on $X_n \coloneqq X \times_A A/I^{n+1}$ can be algebraized to a coherent sheaf on $X$. The theorem has been extended to the case of algebraic spaces by \cite[Theorem 6.3]{ref:knutson} and most recently by \cite[Theorem 1.4]{ref:olsson} and  \cite[Theorem 4.1]{ref:conrad} for algebraic stacks. In all these cases, the proofs rely on Chow's Lemma  and standard d\'{e}vissage techniques.
\medskip

Lieblich's proof of  \cite[Proposition 3.6.1]{ref:lieblich} ultimately reduces to the case of coherent sheaves on algebraic spaces \cite[Theorem 6.3]{ref:knutson}. Therefore, given the results of \cite[Theorem 1.4]{ref:olsson} and \cite[Theorem 4.1]{ref:conrad}, it is reasonable to ask if the proof of \cite[Proposition 3.6.1]{ref:lieblich} generalizes verbatim to the case of algebraic stacks. We now give a rough idea for why this is not the case. Suppose that $X_A$ is an algebraic space, proper and flat over an Artinian local ring $A$. Let $A \to A_0$ be a square-zero thickening, $E_0$  an object of $D^b(\operatorname{Coh}(X_{A_0}))$ that is relatively perfect over $A_0$, and $E \in D^b(\operatorname{Coh}(X_A))$ such that $E \otimes_A^{\operatorname{L}} A_0 \stackrel{\sim}{\to} E_0$. Lieblich constructs a tangent-obstruction theory for deformations of objects in the derived category \cite[Theorem 3.1.1]{ref:lieblich}, crucially relying on the fact that the (small) \'{e}tale site is invariant under infinitesimal thickenings. Lieblich uses this to show that $E_0$ is quasi-isomorphic to a bounded above complex $J_0^{\bullet}$ that lifts to a complex $J^{\bullet}$ satisfying $J^{\bullet} \otimes_A A_0 = J^{\bullet}_0$. 
\medskip

The upshot of this procedure is that if $X$ is an algebraic space over a now complete Noetherian local ring $A$, any adic system of relatively perfect objects $\{E_n\} \in D^b(\operatorname{Coh}(X_n))$ can be replaced by an adic system of \textit{complexes} $\{J_n^{\bullet}\}$. In other words, the transition maps $J_n^{\bullet} \to J_{n-1}^{\bullet}$ are maps of complexes such that $J_n^{\bullet} \otimes_{A_n} A_{n-1} = J_{n-1}^{\bullet}$. Lieblich then completes the proof by algebraizing the \textit{inverse limit} $\lim J_i^{\bullet}$ to a relatively perfect complex on $X$ using \cite[Theorem 6.3]{ref:knutson}. However, if we now wish to work on an \textit{algebraic stack} $\mathscr{X}$, Lieblich's proof does not generalize for the following reason: The lisse-\'{e}tale site is not invariant under infinitesimal thickenings and therefore Lieblich's tangent-obstruction theory cannot be applied to replace an adic system of objects in the derived category with an adic system of complexes. It follows that any attempt to generalize Lieblich's result to algebraic stacks must either develop a deformation theory of complexes on the lisse-\'{e}tale site, or avoid any recourse to deformation theory completely.

 \medskip	

We will take the latter route in this article and show that it is not necessary to represent an adic system of relatively perfect objects in $ D^b(\operatorname{Coh}(\mathscr{X}_n)) $ by an adic system of relatively perfect complexes. Instead,  given an adic system of relatively perfect objects $P_n$, we will directly construct a pseudo-coherent object on $\widehat{\hspace{0ex} \mathscr{X}}$ (the \textit{derived limit} $R\lim P_n$) that recovers each $P_n$ at the finite level. This object will then be algebraized to a pseudo-coherent object $P$ on $\mathscr{X}$ by formal GAGA. Finally, we use the fact that $\mathscr{X} \to \spec A$ is flat, and the fact that $P_0$ is relatively perfect  over $\spec A/I$ to prove that the algebraized object $P$ is relatively perfect over $\spec A$.

\medskip 

The outline of this article is as follows. First, in Section \ref{sec:equivdc}  we provide a very general criterion for when a morphism of ringed sites $f \colon (Y,\mathscr{O}_Y) \to (X,\mathscr{O}_X)$ gives rise to an equivalence of triangulated categories $f^\ast \colon D^-_{\mathscr{A}_X}(X) \to D^-_{\mathscr{A}_Y}(Y)$ (see Section \ref{sec:equivdc} for notation). This will be done  using a slight variant of \cite[Lemma 2.1.10]{ref:LO}. We apply this criterion to prove formal GAGA for pseudo-coherent complexes in Section \ref{section:gagapc}. As a corollary of formal GAGA, we prove Grothendieck's existence theorem for pseudo-coherent complexes in Section \ref{section:existencepc}.  In Section \ref{section:existencep}, we apply the preceding results to prove formal GAGA and  Grothendieck's existence theorem for perfect complexes. Finally, in Section \ref{section:existencerp} we prove Grothendieck's existence theorem for relatively perfect complexes.

\subsection{Recent results in the literature}

 Hall \cite{ref:hall} has proven an existence theorem for pseudo-coherent complexes in a non-Noetherian setting. Unlike previous proofs of the existence theorem, Hall's proof does not rely on d\'{e}vissage to the projective situation. Instead, he proves a very general result \cite[Theorem 6.1]{ref:hall} that is sufficient to imply all existing GAGA results in the setting of $\mathbf{C}$-analytic spaces, rigid-analytic spaces, Berkovich spaces or formal algebraic spaces \cite[Examples 7.5,  7.7 and 7.8]{ref:hall}. At this point, we make the important remark that our work  is \textit{not} a generalization of \cite[Example 7.8]{ref:hall} to the setting of algebraic stacks  because we assume Noetherian hypotheses. Conversely, Hall's method does not generalize to algebraic stacks because a key hypothesis of \cite[Theorem 6.1]{ref:hall} fails in this setting, namely that $R\Gamma(\mathscr{X},-)$ preserves pseudo-coherence. Indeed, the algebraic stack $B\mathbf{Z}/p\mathbf{Z}$ is proper over $\spec \mathbf{F}_p$, but $R\Gamma(B\mathbf{Z}/p\mathbf{Z}, \mathscr{O}_{ B\mathbf{Z}/p\mathbf{Z}})$ has non-zero cohomology in all positive degrees. Therefore,  any generalization of Hall's method to prove non-Noetherian GAGA for stacks should at least include the hypothesis that $\mathscr{X}$ is tame, or admits a good moduli space in the sense of Alper \cite[Definition 4.1]{ref:alper}. In summary, Hall's results are essentially \textit{orthogonal} to ours.
\medskip

Finally we mention results of Halpern-Leistner--Preygel \cite{ref:hlp}. They have established various GAGA results for algebraic stacks in the context of derived algebraic geometry, with a focus on those with affine diagonal. Since the diagonals of the algebraic stacks considered in the present paper are proper, our results only overlap in the case of finite diagonals.

\section{Acknowledgements}

I am grateful to Max Lieblich for suggesting the problem to me, in addition to his generosity with his time. I am very grateful to my advisors Jarod Alper and Ravi Vakil for their guidance, time and many useful comments regarding the project. I am very indebted to Jack Hall for his technical expertise in formal algebraic geometry. This project benefited from numerous helpful discussions with Daniel Bragg, Brian Conrad, Daniel Dore, Alessandro M. Masullo, Siddharth Mathur, Zev Rosengarten, Jesse Silliman and Bogdan Zavyalov. Part of this work was carried out at the mathematics department of the University of Washington. We thank them for their hospitality.

\section{A Preliminary Result} \label{sec:equivdc} 

The goal of this section is to establish a very general result (Theorem \ref{thm:genpc})  that  we will use to deduce formal GAGA for pseudo-coherent complexes. Roughly, given a flat morphism of ringed sites $f \colon (Y,\mathscr{O}_Y)  \to (X,\mathscr{O}_X),$ we are interested in sufficient conditions for there to exist an equivalence of categories 
\[ f^\ast \colon D^-_{\mathscr{A}_X}(X) \stackrel{\sim}{\to} D^-_{\mathscr{A}_Y}(Y).\]
The notation $D^-_{\mathscr{A}_X}(X)$ refers to the following. Write $D(X)$ for the unbounded derived category of $\mathscr{O}_X$-modules, $D(\operatorname{Mod}(\mathscr{O}_X))$. Fix a weak Serre subcategory $\mathscr{A}_X \subseteq \operatorname{Mod}(\mathscr{O}_X)$.  Then $D^-_{\mathscr{A}_X}(X)$ is the full triangulated subcategory of $D(X)$ consisting of bounded above objects whose cohomology sheaves  lie in  $\mathscr{A}_X \subseteq \operatorname{Mod}(\mathscr{O}_X)$ (here we use the fact that $\mathscr{A}_X$ is a weak Serre subcategory in order that $D^-_{\mathscr{A}_X}(X)$ be a full triangulated subcategory of $D(X)$). Analogously, we can define $D^-_{\mathscr{A}_Y}(Y)$. Before we can state our result, we will need a rather technical lemma which is basically \cite[Lemma 2.1.10]{ref:LO}. We will include a proof of it here because the proof in \cite{ref:LO} relies on \cite[Lemma 0.11]{ref:spaltenstein} whose proof is light on details. In addition, our hypotheses are a little different than \cite[Lemma 2.1.10]{ref:LO}.

\begin{definition}
Let $X$ be a (Grothendieck) site. A \textit{prebasis} for the topology on $X$ is a subclass $\mathscr{B}_X \subseteq \operatorname{Ob}(X)$ such that every covering $\{V_i \to V\}$ can be refined to a covering by elements of $\mathscr{B}_X$.
\end{definition}

\begin{example}
If $X$ is an algebraic stack, we can take affines with a smooth morphism to $X$ as a prebasis for the lisse-\'{e}tale topology on $X$.
\end{example}

\begin{lemma} \label{lem:sheaf}
Let $F$ be a presheaf of abelian groups on a site $X$ with a prebasis $\mathscr{B}_X$. Let $F^{\#}$ denote the sheafification of $F$. Suppose for every $U \in \mathscr{B}_X$ that $F({U}) = 0$. Then $F^{\#} = 0$.
\end{lemma}

\begin{proof}
By the universal property of sheafification, it is enough to prove for any sheaf $G$ and morphism of presheaves $\varphi \colon F \to G$ that $\varphi = 0$. Choose any $V \in \operatorname{Ob}(X)$ and section $s \in F(V)$. By assumption, there exists a cover $\{U_i \to V\}$ with $U_i \in \mathscr{B}_X$ such that $s|_{U_i} = 0$. Then $\varphi(s)|_{U_i} = \varphi(s|_{U_i}) = 0$. Since $G$ is a sheaf, $\varphi(s) = 0$. We conclude that $\varphi = 0$ as desired. \end{proof}

The technical lemma that we need is this:

\begin{lemma}\label{lem:LO}
Let $f \colon (Y,\mathscr{O}_Y) \to (X,\mathscr{O}_X)$ be a morphism of ringed sites which admit prebases $\mathscr{B}_X$ and $\mathscr{B}_Y$ respectively. Let $\mathscr{A}_Y \subseteq \operatorname{Mod}(\mathscr{O}_Y)$ be a weak Serre subcategory and fix $\mathscr{G} \in D(Y)$. Suppose the following conditions hold:

\begin{enumerate}

\item For any object $U$ in $ \mathscr{B}_X$, $u_f(U)$ is in $\mathscr{B}_Y$ where $u_f \colon\operatorname{Ob}(X) \to \operatorname{Ob}(Y)$ is the functor on objects associated to the morphism of sites $f$. \label{basispreservation}
\item For any $V \in \mathscr{B}_Y$ and $G\in \mathscr{A}_Y$,  $H^p(V,G) = 0$ for all $p > 0$. \label{coherentvanishing}
\item For any $q \in \mathbf{Z}$, $H^q(\mathscr{G}) \in \mathscr{A}_Y$. \label{coherentcohomology}

\end{enumerate}
Choose any integer $n \in \mathbf{Z}$. Then for any non-negative integer $j_0 $ such that $n + j_0  \geq 0$, the natural map
\[ H^n( Rf_\ast \mathscr{G} )\to H^n (Rf_\ast \tau_{\geq -j_0} \mathscr{G})\]
is an isomorphism.
\end{lemma}

\begin{proof} The exact triangle
\[ Rf_\ast \tau_{ < -j_0} \mathscr{G} \to Rf_\ast \mathscr{G} \to R f_\ast \tau_{\geq -j_0} \mathscr{G}\]
gives rise to an exact sequence in cohomology
\[   H^n(R f_\ast \tau_{ < -j_0}  \mathscr{G}) \to H^n (Rf_\ast  \mathscr{G}) \to H^n(R f_\ast \tau_{\geq -j_0} \mathscr{G}) \to H^{n+1}(Rf_\ast \tau_{< -j_0}\mathscr{G}).\]
We need to show that the left and right-most terms are zero. To do this, first recall that If $(C, \mathscr{O}_C)$ is a ringed site and $K$ an object of $D(\mathscr{O}_C)$, the $n$-th cohomology sheaf of $K$, $H^n(K)$, is the sheafification of the presheaf  $U \mapsto H^n(U,K)$, where $H^n(U,K)$ is the $n$-th cohomology group $H^n( R\Gamma(U,K))$. Therefore, by this and Lemma \ref{lem:sheaf}, it is enough to show for any object $U \in \mathscr{B}_X$ that the following is true.

\begin{claim} \label{claim:properties} The map
\[H^{i}(U, Rf_\ast \mathscr{G}) \to H^{i}(U, Rf_\ast \tau_{\geq -j_0} \mathscr{G})\]
has the following properties. It is:
\begin{itemize}

\item Surjective for $i = n-1$.
\item An isomorphism for $i = n$.
\item Injective for $ i =n+1$.
\end{itemize}
\end{claim}

We show the $i=n$ case is an isomorphism as follows. First, using  assumptions (\ref{coherentvanishing})  (\ref{coherentcohomology}) we can apply \cite[\href{https://stacks.math.columbia.edu/tag/0D6P}{Tag 0D6P}]{ref:stacks} (which relies on the techniques of \cite{ref:spaltenstein} on $K$-injective resolutions) to conclude that the natural map
\[ \mathscr{G} \to R\lim \tau_{\geq -j} \mathscr{G}\]
is an isomorphism, where $j$ runs through all positive integers. Then, since derived pushforward commutes with derived limits \cite[\href{https://stacks.math.columbia.edu/tag/0A07}{Tag 0A07}]{ref:stacks},
\[ H^n(U, Rf_\ast \mathscr{G}) \stackrel{\sim}{\to} H^n(U, R\operatorname{lim} Rf_\ast \tau_{\geq -j} \mathscr{G}). \]
It is now sufficient to prove that
\[ H^n(U, R\operatorname{lim} Rf_\ast \tau_{\geq -j} \mathscr{G}) \stackrel{\sim}{\to} H^n(U, Rf_\ast \tau_{\geq - j_0} \mathscr{G}). \]
To get a hold on the derived limit, we use the Milnor exact sequence
\[ 0 \to R^1\lim H^{n-1}(U,Rf_\ast \tau_{\geq -j} \mathscr{G}) \to H^n(U,R\lim Rf_\ast \tau_{\geq -j} \mathscr{G})) \to  \lim H^n(U, Rf_\ast \tau_{\geq -j} \mathscr{G}) \to 0.  \]
\noindent We will show that
\begin{eqnarray*} R^1\lim H^{n-1}(U,Rf_\ast \tau_{\geq -j} \mathscr{G}) &=& 0,  \\
 \lim H^n(U,Rf_\ast \tau_{\geq -j} \mathscr{G}) &=&  H^n(U, Rf_\ast \tau_{\geq -j_0} \mathscr{G}). \end{eqnarray*}

To this end, consider the exact triangle
\[ H^{-(j_0+k)}(\mathscr{G})[j_0+k] \to \tau_{\geq -(j_0+k)} \mathscr{G} \to \tau_{\geq-(j_0+k) +1} \mathscr{G},\]
where $k$ is a positive integer that we will allow to vary. Applying $R\Gamma(U,-)$ and then taking cohomology, we get an exact sequence
\[  \begin{tikzcd}[row sep=1cm] H^{n+j_0+k }( U, Rf_\ast H^{-(j_0+k)} (\mathscr{G})) \ar{r} & H^{n}(U, Rf_\ast \tau_{\geq -(j_0+k)} \mathscr{G}) \ar{r} & H^{n}(U, Rf_\ast \tau_{\geq -(j_0+k)+1} \mathscr{G})  \ar[out=0, in=180, looseness=1, overlay]{dll} \\
 H^{n+j_0+k+1}(U, Rf_\ast H^{-(j_0+k)} (\mathscr{G}). & & {} \end{tikzcd}  \]
The key observation now is that for $k \geq 1$,
\begin{eqnarray*} 
H^{n+j_0+k }(U, Rf_\ast H^{-(j_0+k)} (\mathscr{G})) &\stackrel{\text{def}}{\equiv}& H^{n+j_0+k }(R\Gamma(U, Rf_\ast H^{-(j_0+k)} (\mathscr{G})))  \\
&=&H^{n+j_0+k }(u_f(U), H^{-(j_0+k)} (\mathscr{G})) \\
&=& 0.
\end{eqnarray*} Indeed by (\ref{basispreservation}), $u_f(U) \in \mathscr{B}_Y$ and the fact that $k \geq 1$ implies $n + j_0 + k >0$. The vanishing of the above cohomology group then follows from assumptions (\ref{coherentvanishing}) and (\ref{coherentcohomology}). Similarly, 
\[ H^{n+j_0+k+1}(U, Rf_\ast H^{-(j_0+k)} (\mathscr{G})) =0\]
 for $k \geq 1$ as well. The conclusion is that 
\[ G_{-j} \coloneqq H^n(U, Rf_\ast \tau_{\geq - j}\mathscr{G})\]
is an inverse system of abelian groups such that $G_{-(j_0+k)} \stackrel{\sim}{\to} G_{-(j_0+k)+1}$ for all $k \geq 1$. By Lemma \ref{lem:abelianlim} (\ref{iso})  below, we conclude that
\[ \lim H^n(U,Rf_\ast \tau_{\geq -j} \mathscr{G}) = H^n(U, Rf_\ast \tau_{\geq -j_0} \mathscr{G}).\]
Reasoning similarly using Lemma \ref{lem:abelianlim} (\ref{ML}),
\[ R^1\lim H^{n-1}(U,Rf_\ast \tau_{\geq -j} \mathscr{G}) = 0.\]
This  finishes the $i=n$ case of Claim \ref{claim:properties}. The rest of the aforementioned claim is proven similarly. \end{proof}

\begin{lemma} \label{lem:abelianlim}
Let $\{G_{-j}\}_{j \in \mathbf{N}}$ be an inverse system of abelian groups
\[\ldots \to G_{-3} \to G_{-2} \to G_{-1}. \]
Suppose there exists $j_0 \geq 1$ such that for all $k \geq 1$, $G_{-(j_0+k)} \to G_{-(j_0+k)+1}$ is surjective. Then the following hold:
\begin{enumerate}[(a)]

\item The system $\{G_{-j}\}_j$ is Mittag-Leffler. \label{ML}
\item The natural map $\lim G_{-j} \to G_{-j_0}$ is surjective. \label{surj}
\item In addition, if the transition maps $G_{-(j_0+k)} \to G_{-(j_0+k)+1)}$ are injective (a fortiori isomorphisms) for $k \geq 1$, then $\lim G_{-j} \to G_{-j_0}$ is an isomorphism. \label{iso}
\end{enumerate}
\end{lemma}

\begin{proof}
Statement (\ref{ML}) follows from the fact that the transition maps are all eventually surjective. For (\ref{surj}), we must show that any  $x \in G_{-j_0}$ can be lifted to $G_{-j_0 +k}$ for all $k \geq 1$. This follows from the assumption that all transition maps  $G_{-(j_0+k)} \to G_{-(j_0+k)+1}$ are surjective. Finally, suppose that $G_{-(j_0+k)} \to G_{-j_0}$ is an isomorphism for all $k \geq 1$. If an element 
\[ \overrightarrow{x} \coloneqq (\ldots, x_{-(j_0+1)}, x_{-(j_0)}, x_{-(j_0-1)},\ldots, x_{-1}) \in  \lim G_{-j}\]
maps to zero in $G_{-j_0}$, then necessarily $x_{-j_0}=0$. By definition of the inverse limit, 
\[ x_{-(j_0-1)}  = \ldots = x_{-1} = 0.\]
By injectivity,  $x_{-(j_0+k)} =0$ for all $k \geq 1$ and therefore $\overrightarrow{x} = 0$. \end{proof}
\medskip

Here is the main result of this section. Recall that a morphism of ringed sites $f: (Y, \mathscr{O}_Y) \to (X,\mathscr{O}_X)$ is said to be \textit{flat} if the ring map $f^{-1} \mathscr{O}_X \to \mathscr{O}_Y$ is flat. In particular, this implies that the pullback $f^\ast \colon \operatorname{Mod}(\mathscr{O}_X) \to \operatorname{Mod}(\mathscr{O}_Y)$ is an exact functor. 

\begin{theorem} Let $f\colon (Y,\mathscr{O}_Y) \to (X,\mathscr{O}_X)$ be a flat morphism of ringed sites which admit prebases $\mathscr{B}_Y$ and $\mathscr{B}_X$ respectively. Fix  weak Serre subcategories $\mathscr{A}_Y \subseteq \operatorname{Mod}(\mathscr{O}_Y)$ and $\mathscr{A}_X \subseteq \operatorname{Mod}(\mathscr{O}_X)$. Suppose that the following conditions hold: \label{thm:genpc}
\begin{enumerate}[(i)]
 
\item The pullback map $f^\ast \colon \operatorname{Mod}(\mathscr{O}_X) \to \operatorname{Mod}(\mathscr{O}_Y)$ restricts to an equivalence of categories \label{heart}
\[ f^\ast \colon \mathscr{A}_X \stackrel{\sim}{\to} \mathscr{A}_Y.\] 
\item For every $F,F' \in \mathscr{A}_X$ and $n \in \mathbf{Z}$, the natural map
\[ \operatorname{Ext}^n_{\mathscr{O}_X}(F,F') \to \operatorname{Ext}^n_{\mathscr{O}_Y}(f^\ast F, f^\ast F')\]
is an isomorphism of abelian groups.\label{ext} 

\item For any object $U$ in $ \mathscr{B}_X$, $u_f(U)$ is in $\mathscr{B}_Y$ where $u_f \colon\operatorname{Ob}(X) \to \operatorname{Ob}(Y)$ is the functor associated to the morphism of sites $f$. \label{objects}
\item For every $G \in \mathscr{A}_Y$ and $V \in \mathscr{B}_Y$, $H^p(V,G) = 0$ for all $p > 0$.\label{vanishing}
\end{enumerate}
Then the pullback $f^\ast \colon D^-(X) \to D^-(Y)$ restricts to an equivalence of categories
\[ f^\ast \colon D^{-}_{\mathscr{A}_X}(X) \stackrel{\sim}{\to} D^-_{\mathscr{A}_Y}(Y).\]
\end{theorem}

\begin{proof}
We will first prove that $f^\ast$ restricts to a fully faithful functor 
\[f^\ast \colon D^b_{\mathscr{A}_X}(X) \to D^b_{\mathscr{A}_Y}(Y).\]
The fact that $f^\ast$ sends bounded objects to bounded objects is clear because $f$ is flat, a fortiori commutes with cohomology. Fix $\mathscr{F} \in  D^b_{\mathscr{A}_X}(X) $. It is enough to prove for any $\mathscr{F}' \in D^b_{\mathscr{A}_X}(X)$ that
\begin{equation} \label{eq:fullfaithful} R\Hom_{\mathscr{O}_{X}}(\mathscr{F}, \mathscr{F}') = R\Hom_{\mathscr{O}_Y}(f^\ast \mathscr{F}, f^\ast \mathscr{F}').\end{equation}
Let us first assume that $\mathscr{F}'$ is concentrated in a single degree, so in particular is in $\mathscr{A}_X$. We will prove the equality above by induction on the length of $\mathscr{F}$. If $\mathscr{F}$ is concentrated in a single degree, then this is condition (\ref{ext}) above. Now suppose that $\mathscr{F}$ is a bounded complex concentrated in the interval $[a,b]$. Consider the exact triangle
\[ H^{a}(\mathscr{F})[-a] \to \mathscr{F} \to \tau_{\geq a+1} \mathscr{F} .\]
Applying $f^\ast$ and then $R\Hom_{\mathscr{O}_Y}(-, f^\ast\mathscr{F}')$ gives a diagram
\[ \begin{tikzcd}  R\Hom_{\mathscr{O}_X}( \tau_{\geq a+1} \mathscr{F}, \mathscr{F}') \ar{r} \ar{d} & R\Hom_{\mathscr{O}_X}(\mathscr{F}, \mathscr{F}') \ar{r} \ar{d} & R\Hom_{\mathscr{O}_X}( H^{a}(\mathscr{F})[-a], \mathscr{F}') \ar{d}  \\ 
 R\Hom_{\mathscr{O}_Y}(f^\ast \tau_{\geq a+1} \mathscr{F}, f^\ast \mathscr{F}') \ar{r} & R\Hom_{\mathscr{O}_Y}(f^\ast \mathscr{F}, f^\ast \mathscr{F}') \ar{r}  & R\Hom_{\mathscr{O}_Y}(f^\ast  H^{a}(\mathscr{F})[-a], f^\ast  \mathscr{F}') .\end{tikzcd} \] 
By induction, the left vertical arrow is an isomorphism. Since all the cohomology sheaves of $\mathscr{F}$ are in $\mathscr{A}_X$, the right vertical arrow is an isomorphism. Therefore the middle vertical arrow is an isomorphism. This shows (\ref{eq:fullfaithful})  in the case that $\mathscr{F}'$ is concentrated in a single degree, and repeating the same argument above shows (\ref{eq:fullfaithful})  in general. Therefore $f^\ast \colon D^b_{\mathscr{A}_X}(X) \to D^b_{Y}(Y)$ is fully faithful.
\medskip

We now prove that 
\[f^\ast \colon D^b_{\mathscr{A}_X}(X) \to D^b_{\mathscr{A}_Y}(Y).\]
is essentially surjective, a fortiori an equivalence of categories. Fix $\mathscr{G} \in D^b_{\mathscr{A}_Y}(Y)$. Suppose that $\mathscr{G}$ is concentrated in the interval $[a,b]$. Consider the exact triangle
\[  (\tau_{\geq a +1 } \mathscr{G})[-1] \to H^a(\mathscr{G})[-a] \to \mathscr{G}.\]
By (\ref{heart}) and the inductive hypothesis, there exists $\mathscr{F}, \mathscr{F}'$ so that
\begin{eqnarray*} f^\ast \mathscr{F} &=& (\tau_{\geq a +1 } \mathscr{G})[-1], \\
f^\ast \mathscr{F}' &=&  H^a(\mathscr{G})[-a].\end{eqnarray*}
By full faithfulness, the morphism  $(\tau_{\geq a +1 } \mathscr{G})[-1] \to H^a(\mathscr{G})[-a] $ is equal to $f^\ast \varphi$ for some $f \colon \mathscr{F} \to \mathscr{F}'$. It follows that $\mathscr{G} \cong f^\ast \operatorname{Cone}(\mathscr{F} \to \mathscr{F}')$ and so we have the equivalence on bounded derived categories as claimed.
\medskip

Finally, we will use Lemma \ref{lem:LO} to reduce to the bounded case as follows. We must show for any $\mathscr{G} \in D^-_{\mathscr{A}_Y}(Y)$ and $\mathscr{F} \in D^-_{\mathscr{A}_X}(X)$ that the counit $f^\ast Rf_\ast \mathscr{G} \to \mathscr{G}$ and  unit $\mathscr{F} \to Rf_\ast f^\ast \mathscr{F}$
are isomorphisms. Equivalently, for any $n \in \mathbf{Z}$ that
\[H^n(f^\ast Rf_\ast \mathscr{G}) \to H^n(\mathscr{G}) \hspace{0.5cm}  \text{and} \hspace{0.5cm} H^n(\mathscr{F}) \to H^n(Rf_\ast f^\ast \mathscr{F})\]
are isomorphisms. Consider first the case of the counit. Fix $n \in \mathbf{Z}$ and choose a non-negative integer $j_0 \geq 0$ such that $n + j_0 \geq 0$. We want to show that the arrow (1) in the diagram below is an isomorphism.
\[ \begin{tikzcd} H^n(f^\ast Rf_\ast \mathscr{G}) \ar{r}{(1)} \ar[equals,d] & H^n(\mathscr{G}) \ar{ddd}{(4)} \\
f^\ast H^n(Rf_\ast \mathscr{G}) \ar{d}{(2)} & {} \\
f^\ast H^n(Rf_\ast \tau_{\geq -j_0} \mathscr{G}) \ar[equals]{d} & {} \\
 H^n(f^\ast Rf_\ast \tau_{\geq -j_0} \mathscr{G}) \ar{r}{(3)}& H^n( \tau_{\geq -j_0}\mathscr{G}) 
\end{tikzcd} \]
The displayed equalities follow from flatness and arrow (2) is an isomorphism from Lemma \ref{lem:LO}. Furthermore, the object $\tau_{\geq -j_0} \mathscr{G}$ is bounded and therefore (3) is an isomorphism. Arrow (4) is trivially an isomorphism and therefore (1) is an isomorphism as desired. Reasoning similarly with the diagram 
\[ \begin{tikzcd} H^n(\mathscr{F}) \ar{r} \ar{dd} & H^n(Rf_\ast f^\ast \mathscr{F}) \ar{d} \\
{} & H^n(Rf_\ast \tau_{\geq -j_0} f^\ast \mathscr{F}) \ar[equals]{d} \\
H^n(\tau_{\geq -j_0} \mathscr{F}) \ar{r} & H^n(Rf_\ast f^\ast \tau_{\geq -j_0} \mathscr{F}) 
\end{tikzcd} \]
allows us to conclude that the unit is an isomorphism. This completes the proof of  Theorem \ref{thm:genpc}. \end{proof}

\section{Formal GAGA  for Pseudo-Coherent Complexes} \label{section:gagapc}

\subsection{Preliminary Definitions} Let $R$ be a ring. Recall that an $R$-module $M$ is said to be finitely presented if there exists an exact sequence $R^{n_{-1}} \to R^{n_0} \to M \to 0$. Alternatively, thinking of $M$ as being the complex $M[0]$, there exists a morphism of complexes
\[  \begin{tikzcd}  \ldots \ar{r} & 0\ar{d}  \ar{r} & R^{n_{-1}} \ar{r} \ar{d}  & R^{n_0} \ar{d} \ar{r} & 0 \ar{d} \ar{r}  & \ldots \\
\ldots \ar{r} &0 \ar{r} & 0 \ar{r} & M \ar{r} & 0\ar{r}  & \ldots \end{tikzcd} \]
with the following properties. 
\begin{itemize}

\item  The map on cohomology is an isomorphism in degree $0$. 
\item The map on cohomology  is surjective in degree $-1$.
\end{itemize} 

\noindent This motivates the following definition:

\begin{definition}[{\cite[\href{https://stacks.math.columbia.edu/tag/064Q}{Tag 064Q}]{ref:stacks}}]\label{def:pc} Let $R$ be a ring. A complex of $R$-modules $P^{\bullet}$ is said to be \textit{$m$-pseudo coherent} (for some $m \in \mathbf{Z}$) if there exists a bounded complex of finite free modules $F^{\bullet}$ and a morphism of complexes  $F^{\bullet} \to P^{\bullet}$ such that  the following hold:
\begin{itemize}
\item $H^i(F^{\bullet}) \to H^i(P^{\bullet})$ is an isomorphism for $i > m$. 
\item $H^i(F^{\bullet}) \to H^i(P^{\bullet})$  and surjective for $i = m$. 
\end{itemize} 
A complex $P^{\bullet}$  is \textit{pseudo-coherent} if it is $m$-pseudo-coherent for every $m \in \mathbf{Z}$. An object $P \in D(R)$ is \textit{pseudo-coherent} if it is quasi-isomorphic to a pseudo-coherent complex.
\end{definition}

\begin{example} An $R$-module $M$ is finite if and only if $M[0] \in D(R)$ is $0$-pseudo-coherent, and finitely presented if and only if $M[0]$ is $-1$-pseudo-coherent. This is clear from the definition above.
\end{example}

\begin{remark}
Let $P$ be an object of $D(R)$. Then $P$ is pseudo-coherent if and only if there exists a quasi-isomorphism $F^{\bullet} \to P$, where $F^{\bullet}$ is a bounded above complex of finite free $R$-modules. This is \cite[\href{https://stacks.math.columbia.edu/tag/064T}{Tag 064T}]{ref:stacks}.
\end{remark}

We will let $D^-_{\operatorname{pc}}(R)$ denote the full triangulated subcategory of $D(R)$ consisting of objects of $D(R)$ quasi-isomorphic to a pseudo-coherent complex. The category $D^-_{\text{pc}}$ is triangulated by \cite[\href{https://stacks.math.columbia.edu/tag/064V}{Tag 064V}]{ref:stacks} and \cite[\href{https://stacks.math.columbia.edu/tag/064X}{Tag 064X}]{ref:stacks}. The definition for rings generalizes to the following situation:

\begin{definition} [{\cite[\href{https://stacks.math.columbia.edu/tag/08FT}{Tag 08FT}]{ref:stacks}}]Let  $(C,\mathscr{O}_C)$ be a ringed site. An object $P \in D(C)$ is said to be $m$-pseudo-coherent, if every $U \in \operatorname{Ob}(C)$ admits a covering $\{U_j \to U\}$, and morphisms $\alpha_j \colon F_j^{\bullet} \to P|_{U_j}$ (with each $F_j^{\bullet}$ a bounded complex of finite projective modules) such that the following hold:
\begin{itemize}
\item $H^i(\alpha_j)$ is an isomorphism for all $i > m$.
\item $H^i(\alpha_j)$ is surjective for $i = m$.
\end{itemize} 
\end{definition} 

 We say that an object $P \in D(C)$ is pseudo-coherent if it is $m$-pseudo-coherent for all $m \in \mathbf{Z}$.  As in the case of rings, we will let $D^-_{\operatorname{pc}}(C)$ denote the full triangulated subcategory of $D(C)$ of objects quasi-isomorphic to a bounded above pseudo-coherent complex. We need an important permanence property of pseudo-coherent objects:

\begin{lemma} Let $f \colon (Y, \mathscr{O}_Y) \to (X,\mathscr{O}_X)$ be a morphism of ringed sites. Then $Lf \colon D(X) \to D(Y)$ restricts to a functor $D^-_{\operatorname{pc}}(X) \to D^-_{\operatorname{pc}}(Y)$.
\end{lemma}

\begin{proof}
See \cite[\href{https://stacks.math.columbia.edu/tag/08H4}{Tag 08H4}]{ref:stacks}.\end{proof}
\medskip

\begin{remark} \label{rem:pc}
If $(C,\mathscr{O}_C)$ is a ringed site with coherent structure sheaf, then $D^-_{\operatorname{pc}}(C) = D^-_{\operatorname{coh}}(C)$, where the latter is the full subcategory of $D(C)$ of objects quasi-isomorphic to a bounded above complex of $\mathscr{O}_C$-modules with coherent cohomology \cite[Expos\'{e} 1, Corollary 3.5]{ref:sga6}. As all the ringed sites we consider in this article have coherent structure sheaf, we will henceforth take pseudo-coherent to mean  ``bounded above with coherent cohomology''.
\end{remark}

\subsection{The GAGA Theorem}
Let $X$ be a locally Noetherian scheme and $I \subseteq \mathscr{O}_X$ a coherent ideal defining a closed subscheme $X_0 \subseteq X$. The \textit{formal completion} of $X$ along $X_0$ (denoted $\widehat{X}_{\operatorname{formal, \zar}}$) is the topologically ringed site $(X_0, \widehat{\mathscr{O}}_{X, \operatorname{formal}})$, where $ \widehat{\mathscr{O}}_{X, \operatorname{formal}}$ is the sheaf of rings
\[  \widehat{\mathscr{O}}_{X_{\operatorname{formal}}}(U_0) \coloneqq \lim H^0(U_n, \mathscr{O}_{U_n}).\]
Here $U_n$ is the unique lifting of the Zariski open $U_0 \to X_0$ to a Zariski open $U_n \to X_n$. For example, if $X = \spec A$, then $\widehat{X}_{\operatorname{formal, \zar}}$ is nothing more than $\operatorname{Spf} \widehat{A}$, where $\widehat{A}$ is the $I$-adic completion of $A$. To state our GAGA theorem, we will need an analogous notion of formal completion for stacks. However, the definition above for schemes does not generalize because the lisse-\'{e}tale site is not invariant under infinitesimal deformations. Instead, we will work with a  surrogate definition of formal completion.

\begin{definition}{\cite[Definition 1.2]{ref:conrad}} The ringed site $\widehat{\hspace{0pt}\mathscr{X}}$ is the lisse-\'{e}tale site of $\mathscr{X}$ equipped  with the sheaf of rings
\[ \widehat{\mathscr{O}}_{\mathscr{X}} \coloneqq \lim \mathscr{O}_{\mathscr{X}}/\mathscr{I}^{n+1},\]
where the limit is taken in the category of lisse-\'{e}tale $\mathscr{O}_{\mathscr{X}}$-modules.
\end{definition}

There is a canonical morphism of ringed sites
\[ \iota\colon \widehat{\hspace{0pt} \mathscr{X}} \to \mathscr{X}\]
which on objects is simply the identity functor. Furthermore, $\iota$ is flat by \cite[Lemma 3.3]{ref:gerdzb}. We now recall the main result of \cite{ref:conrad}.

\begin{theorem} \label{thm:conradgaga}
Let $\mathscr{X}$ be an algebraic stack that is proper over an $I$-adically complete Noetherian ring $A$. Let $\mathscr{I}$ denote the pullback of $I $ to $\mathscr{X}$ and $\widehat{\hspace{0pt}\mathscr{X}}$  the ringed site as defined above with respect to $\mathscr{I}$. Define $\mathscr{X}_n \coloneqq \mathscr{X} \times_A A/I^{n+1}$. Then there are natural equivalences of categories
\[ \operatorname{Coh}(\mathscr{X}) \stackrel{\sim}{\to} \operatorname{Coh}(\widehat{\hspace{0pt}\mathscr{X}}) \stackrel{\sim}{\to} \lim \operatorname{Coh}(\mathscr{X}_n).\]
\end{theorem}

The first arrow is given by sending a coherent sheaf $\mathscr{F}$ on $\mathscr{X}$ to $\iota^\ast \mathscr{F} \simeq \lim \mathscr{F}/ \mathscr{I}^{n+1}\mathscr{F}$, while the second arrow is given by sending a coherent sheaf $\mathscr{G}$ on $\widehat{\hspace{0pt}\mathscr{X}}$ to the adic system $(\mathscr{G}/\mathscr{I}^{n+1} \mathscr{G})_n$.
\medskip

\begin{remark} \label{rem:formalgaga}
In the case that $\mathscr{X} = \spec A$, we have $\operatorname{Coh}(\widehat{\spec A}) = \operatorname{Coh}(\operatorname{Spf} A)$ by \cite[Remark 1.6]{ref:conrad}. Therefore, the proposition above says that 
\[ \operatorname{Coh}(\spec A) \stackrel{\sim}{\to} \operatorname{Coh}(\operatorname{Spf} A) \stackrel{\sim}{\to} \lim \operatorname{Coh}(\spec A/I^{n+1}).\]
It is tempting to think that given an adic system of finite $A/I^{n+1}$-modules $M_n$, the corresponding coherent sheaf on $\spec A$ is $\lim \widetilde{M_n}$. This however is false.  Let $A$ be a complete DVR with uniformizer $\varpi$. The sheaf $\lim \widetilde{A/\varpi^n   }$ on $\spec A$ is not quasi-coherent because its global sections are $A$ while its value on the basic open $D(\varpi)$ is zero. The issue is that the functor sending an $A$-module  $M$ to its associated quasi-coherent sheaf $\widetilde{M} $ does not commute with limits. However, notice that $\lim \widetilde{A/\varpi^n   }$  \textit{is} a coherent sheaf on $\operatorname{Spf} A$. 
\end{remark}

We can now prove the following theorem:

\begin{theorem}[Formal GAGA for pseudo-coherent complexes] \label{thm:gagapc}
Let $\mathscr{X}$ be an algebraic stack that is proper over an $I$-adically complete Noetherian ring $A$. The flat morphism of sites $\iota \colon \widehat{\hspace{0pt}\mathscr{X}} \to \mathscr{X}$ induces an equivalence of categories
\[ \iota^\ast \colon D^-_{\operatorname{coh}}(\mathscr{X}) \stackrel{\sim}{\to} D^-_{\operatorname{coh}}(\widehat{\hspace{0pt}\mathscr{X}}).\]
\end{theorem}

\begin{proof}
We will apply Theorem \ref{thm:genpc} to the morphism of sites $\iota \colon \widehat{\hspace{0pt}\mathscr{X}} \to \mathscr{X}$. We will take smooth morphisms $U \to \mathscr{X}$, where $U$ is affine as a prebasis for the topology on $\mathscr{X}$, and similarly for $\widehat{\hspace{0pt}\mathscr{X}}$ (the underlying sites of $\mathscr{X}$ and $\widehat{\hspace{0pt}\mathscr{X}}$ are the same). The weak Serre subcategories in question are $\operatorname{Coh}(\widehat{\hspace{0pt}\mathscr{X}})$ \cite[Corollary 1.7]{ref:conrad} and $\operatorname{Coh}(\mathscr{X})$\footnote{ The careful reader will note here that it is important to use the lisse-\'{e}tale topology. In the big fppf site, the inclusion functor $\operatorname{Coh}(\mathscr{X}) \subseteq \operatorname{Mod}(\mathscr{O}_{\mathscr{X}})$ is not exact.}. Conditions (\ref{heart}) and (\ref{ext}) are respectively Theorem \ref{thm:conradgaga} above and \cite[Lemma 4.3]{ref:conrad}. Condition (\ref{vanishing}) is Lemma \ref{lem:vanishinget} below. Condition (\ref{objects}) is trivial because the functor on objects $u_{\iota}\colon  \operatorname{Ob}(\mathscr{X}) \to \operatorname{Ob}(\widehat{\hspace{0pt}\mathscr{X}})$ is the identity functor. All the conditions of Theorem \ref{thm:genpc} are satisfied and we are done. \end{proof}

\section{Grothendieck's Existence Theorem for Pseudo-Coherent Complexes} \label{section:existencepc}

\subsection{Preparations}
Let $j \colon \mathscr{Y}\to \mathscr{X}$ be a  morphism of algebraic stacks. For any sheaf $F$ on $\mathscr{Y}$ and object $U$ of $\mathscr{X}_{\liset}$,  recall that $j_\ast F$ is the sheaf on $\mathscr{X}_{\liset}$ whose value on an object $U$ of $\mathscr{X}_{\liset}$ is $F(\mathscr{Y} \times_{\mathscr{X}} U).$ In addition,  also recall that for any sheaf $G$ on $\mathscr{X}_{\liset}$, $j^{-1}G$ is the sheafification of the presheaf whose value on $V \in \mathscr{Y}_{\liset}$ is 
\[ \varinjlim_{ V \to U}  F(U),\] 
where the limit is taken over all $U \in \mathscr{X}_{\liset}$ that fit into a commutative square 
\[ \begin{tikzcd} V \ar{r} \ar{d} & U \ar{d} \\ \mathscr{Y} \ar{r} & \mathscr{X}. \end{tikzcd} \] 
The functor $j^{-1}$ is left adjoint to $j_\ast$. However, it was observed by Gabber and Behrend that the functor $j^{-1}$ is \textit{not} exact (even if $\mathscr{X}$ and $\mathscr{Y}$ are schemes!) and therefore $j$ does not induce a morphism of ringed sites $\mathscr{Y}_{\liset} \to \mathscr{X}_{\liset}$. Hence, it is a  subtle question to construct a derived pullback on the level of derived categories of \textit{sheaves of modules} 
\[ Lj^\ast \colon D({\mathscr{X}}) \to D({\mathscr{Y}}).\]  On the other hand, the \textit{derived pushforward} \begin{equation} \label{eq:derivedpushforward} Rj_\ast \colon D(\mathscr{Y}) \to D(\mathscr{X}) \end{equation} always exists for formal reasons.  Indeed, 
\[ j_\ast \colon \operatorname{Mod}(\mathscr{O}_{\mathscr{Y}}) \to \operatorname{Mod}(\mathscr{O}_{\mathscr{X}}  )\]
 is a left exact functor between Grothendieck abelian categories, and therefore $Rj_\ast$ exists by \cite[Corollary 3.14]{ref:serpe}. Of course, work must be done to show that $Rj_\ast$ has good properties such as a Leray spectral sequence for cohomology, precisely because we do not know that $j_\ast$ preserves injectives\footnote{If $j^{-1}$ were exact, then $j_\ast$ would be right adjoint to an exact functor and therefore preserve injectives.}. Indeed, this is the content of \cite{ref:sheaves}.
 \medskip
 
To remedy the issue of the non-functoriality of the lisse-\'{e}tale site, Hall and Rydh \cite{ref:hallrydh} (building on work of \cite{ref:LO} on unbounded cohomological descent) have defined a surrogate functor 
\[ L(j_{\qc})^\ast \colon D_{\qcoh}({\mathscr{X}}) \to D_{\qcoh}(\mathscr{Y}).\]
We briefly recall the construction of this functor in the case that $j$ is representable, as this is the only situation we need. The general situation is no more difficult. Let $j \colon \mathscr{Y} \to \mathscr{X}$ be a representable morphism of algebraic stacks. Let $p: V \to \mathscr{X}$ be a smooth surjection from an algebraic space. Define $U \coloneqq V \times_\mathscr{X} \mathscr{Y}$ (which is an algebraic space by representability of $j$), and let 
$q \colon U \to \mathscr{Y}$ denote the corresponding smooth surjective morphism. 
Let $p_{\bullet} \colon V_\bullet \to \mathscr{X}$ and $q_\bullet \colon U_\bullet \to \mathscr{Y}$ denote the corresponding morphisms of simplicial algebraic spaces. Associated to each of these simplicial  algebraic spaces are certain ringed topoi $U^+_{\bullet, \text{\'{e}t}}$ and $ V^+_{\bullet, \text{\'{e}t}}$  \cite[Section 1.1]{ref:hallrydh}. Similarly, in the lisse-\'{e}tale topology we have  $U^+_{\bullet, \text{lis-\'{e}t}}$ and $V^+_{\bullet, \text{lis-\'{e}t}}$. These ringed topoi, together with the original $\mathscr{X}$ and $\mathscr{Y}$ sit in a $2$-commutative diagram 
\begin{equation} \label{eq:covering} \begin{tikzcd} \mathscr{Y} \ar{d}{j} &  U^+_{\bullet, \text{lis-\'{e}t}}  \ar[swap]{l}{q_\bullet^+}   \ar{d}{{j}^+_{\bullet,    \text{lis-\'{e}t}  }  } \ar{r}{\text{res}_U} & U^+_{\bullet, \text{\'{e}t}} \ar{d}{{j}^+_{\bullet, \text{\'{e}t}}   }   \\
\mathscr{X} & \ar{l}{p^+_\bullet} V^+_{\bullet, \text{list-\'{e}t}}  \ar[swap]{r}{\text{res}_V} & V^+_{\bullet, \text{\'{e}t}} .\end{tikzcd} \end{equation}
\medskip

The key idea to defining $L(j_{\qc})^\ast$ now is the following. By \cite[Ex. 2.2.5]{ref:LO}, we have equivalences of categories
\begin{eqnarray} \label{eq:un1} \begin{tikzcd}  D_{\qcoh}(\mathscr{Y}) & \ar[swap]{l}{ (q^+_{\bullet})_\ast} D_{\qcoh}( U^+_{\bullet, \text{lis-\'{e}t}}) \ar{r}{(\text{res}_U)_\ast }&  D_{\qcoh}( U^+_{\bullet, \text{\'{e}t}})    \end{tikzcd}    \\
\label{eq:un2} \begin{tikzcd}  D_{\qcoh}(\mathscr{X}) & \ar[swap]{l}{ (p^+_{\bullet})_\ast} D_{\qcoh}( V^+_{\bullet, \text{lis-\'{e}t}}) \ar{r}{(\text{res}_V)_\ast }&  D_{\qcoh}( V^+_{\bullet, \text{\'{e}t}}) .  \end{tikzcd}     \end{eqnarray}
On the other hand, functoriality of the \'{e}tale topos gives a pullback  $L (j^+_{\bullet, \text{\'{e}t}} )^\ast \colon D(V^+_{\bullet, \text{\'{e}t}}) \to D(U^+_{\bullet, \text{\'{e}t}})$ which restricts to $L ({j}^+_{\bullet, \text{\'{e}t}} )^\ast \colon D_{\qcoh} (V^+_{\bullet, \text{\'{e}t}}) \to D_{\qcoh}(U^+_{\bullet, \text{\'{e}t}})$. This allows us to define the functor $L(j_{\qc})^\ast$ as 
\[ L(j_{\qc})^\ast \coloneqq  R( q^+_{\bullet})_\ast \circ L( \operatorname{res}_U)^\ast  \circ     L ({j}^+_{\bullet, \text{\'{e}t}} )^\ast \circ  R(\operatorname{res}_V)_\ast  \circ L(p^+_{\bullet})^\ast.\]
It is readily checked that the definition of $L(j_{\qc})^\ast$ is independent of the choice of smooth cover of $\mathscr{X}$. We mention two important properties of this functor. The first property is that if $j$ is flat, then given any object $P \in D_{\qcoh}(\mathscr{X})$ and integer $n \in \mathbf{Z}$, we have $H^n(L(j_{\qc})^\ast P) \cong j^\ast H^n(P)$. The second property is the following. Suppose that $\mathscr{Y},\mathscr{X}$ are Deligne-Mumford stacks. The \'{e}tale topos is functorial and therefore there is a pullback  $Lj^\ast \colon D(\mathscr{X}_{\et}) \to D(\mathscr{Y}_{\et})$. Since the lisse-\'{e}tale and \'{e}tale topoi agree for Deligne-Mumford stacks, it follows we have a pullback $Lj^\ast \colon D(\mathscr{X}) \to D(\mathscr{Y})$. This pullback restricts to $D_{\qcoh}({\mathscr{X}}) \to D_{\qcoh}(\mathscr{Y})$ and agrees with $L(j_{\qc})^\ast$. 
\medskip

The following results will be used to prove the existence theorem in the next subsection.

\begin{proposition}
Let $ j \colon \mathscr{Y} \to \mathscr{X}$ be a morphism of locally Noetherian algebraic stacks. Then the functor $L(j_{\qc})^\ast \colon D_{\qcoh}(\mathscr{X}) \to D_{\qcoh}(\mathscr{Y})$ restricts to a functor  on pseudo-coherent complexes 
\[ L(j_{\qc})^\ast \colon D^-_{\coh}(\mathscr{X}) \to D^-_{\coh}(\mathscr{Y}).\]
\end{proposition}

\begin{proof}
From the construction of $L(j_{\qc})^\ast$, we see that it is constructed ``locally" at the level of the \'{e}tale topology, where it is given as a pullback arising from a morphism of ringed topoi. Since pseudo-coherent complexes are preserved under arbitrary pullback, we are done.
\end{proof}

\begin{proposition} \label{prop:qcliset}
Let $j \colon \mathscr{Y} \to \mathscr{X}$ be a representable, quasi-compact and quasi-separated morphism of algebraic stacks.  Then the restriction of $Rj_\ast$ to $D_{\qcoh}( {\mathscr{Y}})$ factors through $D_{\qcoh}({\mathscr{X}})$. Furthermore in this situation, $Rj_\ast = R(j_{\qc})_\ast$ where $Rj_\ast$ is the derived pushforward (\ref{eq:derivedpushforward}), and therefore we have an adjoint pair
\[  Rj_\ast \colon D_{\qcoh}(\mathscr{Y}) \leftrightarrows D_{\qcoh}(\mathscr{X})      \colon L(j_{\qc})^\ast.\]
\end{proposition}

\begin{proof}
A representable morphism is concentrated by \cite[Lemma 2.5 (3)]{ref:hallrydh}. The result follows from \cite[Theorem 2.6 (2)]{ref:hallrydh}.
\end{proof}

\begin{proposition} \label{prop:jail}
Let $j \colon \mathscr{Y} \to \mathscr{X}$ be a closed immersion of locally Noetherian algebraic stacks. Then there is a natural equivalence of categories
\[ \overline{j}^\ast \colon D^-_{\operatorname{coh}}(\mathscr{X}, j_\ast \mathscr{O}_{\mathscr{Y}}) \stackrel{\sim}{\to} D^-_{\operatorname{coh}}(\mathscr{Y})\]
\end{proposition}

\begin{proof}
Pick a a diagram as in (\ref{eq:covering}). Then the morphism of ringed topoi $j^+_{\bullet, \text{\'{e}t}} \colon U^+_{\bullet, \text{\'{e}t}} \to V^+_{\bullet, \text{\'{e}t}}  $ factors as 
\[  U^+_{\bullet, \text{\'{e}t}} \stackrel{\overline{j}^+_{\bullet, \text{\'{e}t}}}{\longrightarrow} (V^+_{\bullet, \text{\'{e}t}}, (j^+_{\bullet, \text{\'{e}t}} )_\ast \mathscr{O}_{  U^+_{\bullet, \text{\'{e}t}}  } ) \stackrel{k^+_{\bullet, \text{\'{e}t}}    }{\longrightarrow} V^+_{\bullet, \text{\'{e}t}}  .   \]
Now in \cite[Corollary 2.7]{ref:hallrydh}, it is shown that 
\[ (\overline{j}^+_{\bullet, \text{\'{e}t}})^\ast \colon D_{\qcoh}(V^+_{\bullet, \text{\'{e}t}}, (j^+_{\bullet, \text{\'{e}t}})_\ast \mathscr{O}_ {U^+_{\bullet, \text{\'{e}t}}} ) \to D_{\qcoh}(U^+_{\bullet, \text{\'{e}t}})      \]
is an equivalence of categories. Therefore by (\ref{eq:un1}) and (\ref{eq:un2}) this induces an equivalence $\overline{j}^\ast$ of triangulated categories
\[ \overline{j}^\ast \colon D_{\operatorname{qcoh}}(\mathscr{X}, j_\ast \mathscr{O}_{\mathscr{Y}}) \stackrel{\sim}{\to} D_{\operatorname{qcoh}}(\mathscr{Y}).\] 
To complete the proof, we must show that $\overline{j}^\ast$ descends to an equivalence at the level of pseudo-coherent complexes. From the construction of $\overline{j}$, it is clear that it restricts to a fully faithful functor
\[ \overline{j}^\ast \colon D^-_{\operatorname{coh}}(\mathscr{X}, j_\ast \mathscr{O}_{\mathscr{Y}}) \stackrel{}{\to} D^-_{\operatorname{coh}}(\mathscr{Y}).\]
The formation of $\overline{j}^\ast$ commutes with cohomology and therefore to conclude it is enough to show the following. Let $F$ be any quasi-coherent sheaf on $(\mathscr{X}, j_\ast \mathscr{O}_{\mathscr{Y}})$. If  $\overline{j}^\ast F$ is coherent, then the same is true of $F$. Now the question of coherence is local and therefore we may reduce to the case that $\mathscr{X},\mathscr{Y}$ are affine schemes with the lisse-\'{e}tale topology. By descent, we reduce to the setting of the Zariski topology from which the result is clear. 
\end{proof}

\subsection{Grothendieck's Existence Theorem for Pseudo-Coherent Complexes}
Let $X$ be a scheme that is proper over an $I$-adically complete Noetherian ring $A$. Recall in this case that the existence theorem for coherent sheaves follows immediately from formal GAGA. Indeed, the theory of formal schemes tells us that the reduction map from $\operatorname{Coh}(\widehat{X})$ to the category of adic systems of coherent sheaves $F_n$ on $X_n \coloneqq X \times_A A/I^{n+1}$ is an equivalence of categories. However, in the situation of complexes on an algebraic stack $\mathscr{X}$, the non-functoriality of the lisse-\'{e}tale site means that extra care must be taken.  There are three issues we must clarify:
\begin{enumerate}
 
\item The notion of an adic system of complexes on $\mathscr{X}_n$.
\item The notion of an adic system of complexes on $\mathscr{X}_n$ algebraizing to a complex on $\mathscr{X}$.
\item The definition of a reduction functor $D^-_{\coh} (\widehat{\mathscr{\hspace{0ex} X}}) \to D^-_{\coh}(\mathscr{X}_n)$.
\end{enumerate} 
We address these as follows:

\begin{enumerate}
 
 \item For $\mathscr{X}$ an algebraic stack over an $I$-adically complete Noetherian ring $A$, define $\mathscr{X}_n \coloneqq \mathscr{X} \times_A A/I^{n+1}$. Let 
\begin{eqnarray*} i_{n,n+1} \colon &\mathscr{X}_n& \hookrightarrow \mathscr{X}_{n+1}\\
 j_n \colon &\mathscr{X}_n& \hookrightarrow \mathscr{X}\end{eqnarray*}
denote the canonical closed immersions of algebraic stacks. By Proposition \ref{prop:qcliset} $L(i_{n,n+1,\qc})^\ast$ is left adjoint to $R(i_{n,n+1})_\ast$. A \textit{system} of pseudo-coherent complexes $P_n$ on $\mathscr{X}_n$ is the data for every $n \geq 0$ an object $P_n \in D^-_{\operatorname{coh}}(\mathscr{X}_n)$ with maps
\begin{equation}  P_{n+1} \to R(i_{n, n+1})_\ast P_n. \label{eq:pushforward} \end{equation}
The system is said to be \textit{adic} if (\ref{eq:pushforward}) is adjoint to an isomorphism
\[ L({i_{n, n+1,\qc}})^\ast P_{n+1} \stackrel{\sim}{\to} P_n.\]

\item We say that an adic system of pseudo-coherent complexes $P_n$ on $\mathscr{X}_n$ \textit{algebraizes} to a pseudo-coherent complex $P$ on $\mathscr{X}$ if there exists an object $P \in D^-_{\coh}(\mathscr{X})$ such that $L(j_{n,\qc})^\ast P \simeq P_n$. 

\item  By Proposition  \ref{prop:jail} that there is a natural equivalence of categories
\[ \overline{j}_n^\ast \colon D^-_{\coh}(\mathscr{X}, j_{n,\ast} \mathscr{O}_{\mathscr{X}_n}  ) \stackrel{\sim}{\to} D^-_{\coh}(\mathscr{X}_n).\]
Therefore, it is enough to define a functor 
\[ D^-_{\coh}(\widehat{\hspace{0ex} \mathscr{X}}) \to D^-_{\coh}(\mathscr{X}, j_{n,\ast} \mathscr{O}_{\mathscr{X}_n}).\] However, because the underlying sites of $\widehat{\hspace{0ex} \mathscr{X}}$ and $(\mathscr{X}, j_{n,\ast} \mathscr{O}_{\mathscr{X}_n})$ are the same, we can simply define this functor by extension of scalars. In summary,  we define $D^-_{\coh} (\widehat{\mathscr{\hspace{0ex} X}}) \to D^-_{\coh}(\mathscr{X}_n)$ as the composition 
\begin{equation} \label{eq:relationship} \begin{tikzcd} D^-_{\coh}(\widehat{\hspace{0ex} \mathscr{X}}) \ar{rrr}{Q \mapsto Q \otimes^{\operatorname{L}}_{\widehat{\mathscr{O}}_{\mathscr{X}}} j_{n,\ast} \mathscr{O}_{\mathscr{X}_n}} &&&  D^-_{\coh}(\mathscr{X}, j_{n,\ast}\mathscr{O}_{\mathscr{X}_n}) \ar{r}{\overline{j}_n^\ast} &  D^-_{\coh}(\mathscr{X}_n). \end{tikzcd} \end{equation}

\end{enumerate}

\noindent The proposition below makes precise the relationship between (\ref{eq:relationship})
and  the pullback \[ L(j_{n,\qc})^\ast \colon D^-_{\coh}(\mathscr{X}) \to D^-_{\coh}(\mathscr{X}_n).\] obtained from \cite{ref:hallrydh}.
 
 \begin{proposition} \label{prop:commute}
Let $j \colon \mathscr{Y} \to \mathscr{X}$ be a closed immersion of locally Noetherian algebraic stacks. There is a 2-commutative diagram

 \[ \begin{tikzcd}
D_{\qcoh}(\mathscr{X}) \ar{rrr}{L(j_{\qc})^\ast} &  && D_{\qcoh}(\mathscr{Y}) \\
D^-_{\coh}(\mathscr{X}) \ar{r}  \ar[hook]{u} & D^-_{\coh}(\mathscr{X}, \widehat{\mathscr{O}}_{\mathscr{X}}) \ar{r} & D^-_{\coh}(\mathscr{X}, j_\ast \mathscr{O}_{\mathscr{Y}}) \ar{r}{\overline{j}^\ast}&  D^-_{\coh}(\mathscr{Y}) \ar[hook]{u} \\
\end{tikzcd} \]
where $\widehat{\mathscr{O}}_{\mathscr{X}}$ is the completion of $\mathscr{O}_{\mathscr{X}}$ with respect to the coherent ideal defining $\mathscr{Y}$, and the functor $\overline{j}^\ast$ is an equivalence by Proposition \ref{prop:jail}. The functor $D^-_{\coh}(\mathscr{X}) \to D^-_{\coh}(\mathscr{X}, \widehat{\mathscr{O}}_{\mathscr{X}}) $ is given by  extending scalars $P \mapsto P \otimes_{\mathscr{O}_{\mathscr{X}}} \widehat{\mathscr{O}}_{\mathscr{X}}.$ Similarly, $D^-_{\coh}(\mathscr{X}, \widehat{\mathscr{O}}_{\mathscr{X}}) \to D^-_{\coh}(\mathscr{X}, j_\ast \mathscr{O}_{\mathscr{Y}})$ is given by $Q \mapsto Q \otimes^{\operatorname{L}}_{\widehat{\mathscr{O}}_{\mathscr{X}}}  j_\ast \mathscr{O}_{\mathscr{Y}}.$

\end{proposition}

\begin{proof}
Pick a a diagram as in (\ref{eq:covering}). Then we know that $L(j_{\qc})^\ast$ is given locally as $j^+_{\bullet, \text{\'{e}t}} \colon U^+_{\bullet, \text{\'{e}t}} \to V^+_{\bullet, \text{\'{e}t}}.$ As observed in Proposition \ref{prop:jail}, this factors as 
\[  U^+_{\bullet, \text{\'{e}t}} \stackrel{\overline{j}^+_{\bullet, \text{\'{e}t}}}{\longrightarrow} (V^+_{\bullet, \text{\'{e}t}}, (j^+_{\bullet, \text{\'{e}t}} )_\ast \mathscr{O}_{  U^+_{\bullet, \text{\'{e}t}}  } ) \stackrel{k^+_{\bullet, \text{\'{e}t}}    }{\longrightarrow} V^+_{\bullet, \text{\'{e}t}}  .   \]
The result follows.
\end{proof}

\begin{theorem}[Grothendieck existence for pseudo-coherent complexes] \label{thm:existencepc}
Let $\mathscr{X}$ that be an algebraic stack that is proper over an  $I$-adically complete Noetherian ring $A$. Define $\mathscr{X}_n \coloneqq \mathscr{X} \times_A A/I^{n+1}$. Let $P_n \in D^-_{\operatorname{coh}}(\mathscr{X}_n)$ be an adic system of pseudo-coherent complexes. Then there exists a pseudo-coherent complex $P \in D^-_{\operatorname{coh}}(\mathscr{X})$ such that $L(j_{n,\qc})^\ast P \simeq P_n$.
\end{theorem}

\begin{proof}
By formal GAGA (Theorem \ref{thm:gagapc}), Proposition \ref{prop:jail} and Proposition \ref{prop:commute}, we reduce to showing the following. Let $Q_n \in D^-_{\coh}(\mathscr{X}, j_{n,\ast} \mathscr{O}_{\mathscr{X}_n})$ be an adic system of pseudo-coherent complexes. Then there exists a pseudo-coherent complex $Q$ on $D^-_{\coh}(\widehat{\hspace{0ex} \mathscr{X}})$ such that 
\[Q \otimes_{\widehat{\mathscr{O}}_{\mathscr{X}}}^{\operatorname{L}}  j_{n,\ast} \mathscr{O}_{\mathscr{X}_n} = Q_n.\] Motivated by Remark \ref{rem:formalgaga}, we construct $Q$ as the derived limit $R\lim Q_k$. First, we show that $R\lim Q_k$ is bounded above with coherent cohomology as follows. As the question of coherence is local, we may assume that $\mathscr{X} = \spec B$ equipped with the lisse-\'{e}tale topology. By Proposition \ref{prop:pctopapp}, Corollary \ref{cor:appzaret} and Remark \ref{cor:applisetet}, we may assume that  each $Q_k$ is a pseudo-coherent object on the ringed site $(\spec B, j_{k,\ast} \mathscr{O}_{\spec B_k})$ in the Zariski topology. In fact, since $\operatorname{Coh}(\spec B, j_{k,\ast} \mathscr{O}_{\spec B_k}) \simeq \operatorname{Coh}(\spec B_k)$, we may reduce to the following situation. We have an adic system of pseudo-coherent objects $Q_k$ on $\spec B_k$ and we must prove that $R\lim Q_k$ is pseudo-coherent on $(\spec B, \widehat{\mathscr{O}}_{\spec B})$.

\medskip

Choose a bounded above complex of finite free $\mathscr{O}_{B_0}$-modules $F_0^{\bullet}$  and a quasi-isomorphism $F_0^{\bullet} \to Q_0$. By  \cite[\href{https://stacks.math.columbia.edu/tag/0BCB}{Tag 0BCB}]{ref:stacks} there is a bounded above complex of finite free $\mathscr{O}_{B_1}$-modules $F_1^{\bullet}$, a quasi-isomorphism $F_1^{\bullet} \to Q_1$ and a map of complexes(!) $F_1^{\bullet} \to F_0^{\bullet}$ lying over $Q_1 \to Q_0$ such that $F_1^{\bullet} \otimes_{B_1} B_0 = F_0^{\bullet}$. Continuing this procedure, we get an adic system of complexes $\{F_k^{\bullet}\}$ such that
\[ R\lim Q_k \stackrel{\sim}{\to} R\lim F_k^{\bullet}\]
in $D( \spec B, \widehat{\mathscr{O}}_{\spec B})$. For a fixed index $i$, the proof of Lemma \ref{lem:vanishinget} shows that $R\lim F_k^i = \lim F_k^i$ and therefore  $R\lim Q_k = \lim F_k^{\bullet}$. Now $F_0^{\bullet}$ is bounded above and by Nakayama's lemma, the cohomology of each $F_k^\bullet$ lives in the same interval as $F_0^\bullet$. This proves that  $\lim F_k^{\bullet}$ is  bounded above. Furthermore for each $i$, $\{ F_k^i\}$ is adic and therefore $\lim F_k^i$ is a coherent sheaf on $(\spec B, \widehat{\mathscr{O}}_{\spec B})$ by the theory of formal schemes. It follows that $\lim F_k^{\bullet}$ has coherent cohomology sheaves and therefore $R\lim Q_k$ is pseudo-coherent.
\medskip

Finally, we must show that $R\lim Q_k \otimes^{\operatorname{L}}_{\widehat{\mathscr{O}}_{\mathscr{X}}}  j_{n,\ast}\mathscr{O}_{\mathscr{X}_n} \stackrel{\sim}{\to} Q_n$. By Corollary \ref{cor:appsheafification} it is enough to prove for every smooth morphism $\spec B \to \mathscr{X}$ that 
\[R\Gamma(\spec B, R\lim Q_k) \otimes_{\widehat{B} }^{\operatorname{L}} B_n \stackrel{\sim}{\to} R\Gamma(\spec B, Q_n).\]
Now $R\Gamma$ commutes with $R\lim$ and therefore the result follows from \cite[\href{https://stacks.math.columbia.edu/tag/0CQF}{Tag 0CQF}]{ref:stacks}.
\end{proof}

\section{Formal GAGA and Grothendieck's Existence Theorem for Perfect Complexes} \label{section:existencep}

\subsection{Preliminary Definitions}
\begin{definition} Let $R$ be a Noetherian ring. A complex of $R$-modules $P^\bullet$ is said to be \textit{strictly perfect} if it is bounded  with each $P^i$ a finite projective module. Furthermore, we say that an object $P \in D(R)$ is  \textit{perfect} if it is quasi-isomorphic to a strictly perfect complex.
\end{definition}

\begin{definition} Let $R$ be a ring. An object $K \in D(R)$, is said to have finite Tor dimension if it has Tor-amplitude in some interval $[a,b]$. That is, for any $R$-module $M$, $ H^i(K \otimes_R^\text{L} M ) = 0$ for all $ i \notin [a,b]$.
\end{definition}

\begin{remark} An object $P \in D(R)$ is perfect if and only if it is pseudo-coherent and of finite Tor dimension \cite[\href{https://stacks.math.columbia.edu/tag/0658}{Tag 0658}]{ref:stacks}.\end{remark}

 As in the case of pseudo-coherence, the notions of being perfect and having finite Tor-dimension generalize to any ringed site $(C, \mathscr{O}_C   )$. A complex $P^{\bullet}$ of $\mathscr{O}_C$-modules is said to be \textit{perfect} if for any object $U$ there is a cover $\{U_i \to U\}$, strictly perfect complexes $P_i^{\bullet}$ and  quasi-isomorphisms of complexes $\alpha_i \colon P_i^{\bullet} \to P^{\bullet}|_{U_i}$. An object $P \in D(\mathscr{O}_C)$ is perfect if it is quasi-isomorphic to a perfect complex. Similarly, an object $K \in D(\mathscr{O}_C)$ has finite Tor-dimension if there is an interval $[a,b]$ such that $H^i(K \otimes_{\mathscr{O}_C}^{\text{L}} F) =0$ for all $\mathscr{O}_C$-modules $F$.
 
 \begin{remark}
 In the case that $C$ is the Zariski site of an affine scheme, and $K$ an object of $D_{\operatorname{qcoh}}(\mathscr{O}_C)$,  the two notions of being of finite Tor dimension (rings and ringed sites) agree  \cite[\href{https://stacks.math.columbia.edu/tag/08EA}{Tag 08EA}]{ref:stacks}.  This is not a priori obvious because in the definition of finite Tor dimension for sites, the test module $F$ is \textit{any} $\mathscr{O}_C$-module, not necessarily quasi-coherent.
 \end{remark}

\subsection{Affine Formal GAGA for Perfect Complexes} The goal of this subsection is to prove a theorem comparing perfect complexes on the Zariski and formal spectra of a $J$-adically complete Noetherian ring $B$. This will be used in the next subsection where we prove the  formal GAGA theorem for perfect complexes on a proper algebraic stack $\mathscr{X}.$ In order to state this theorem precisely, we need to recall some facts about coherent sheaves on  affine formal schemes. Let $B$ be a Noetherian ring, $J$ an ideal of $B$ and $\widehat{B}$ the $J$-adic completion of $B$. It is a fact that the canonical  flat map of ringed spaces $i \colon \operatorname{Spf} \widehat{B} \to \spec \widehat{B}$ gives rise to an equivalence of categories
\[ i^\ast \colon \operatorname{Coh}(\spec \widehat{B}) \stackrel{\sim}{\to} \operatorname{Coh}( \operatorname{Spf} \widehat{B}).\]

\noindent Furthermore, the inverse to $i^\ast$ is given by the global sections functor $\Gamma(\operatorname{Spf} \widehat{B}, -)$ (by abuse of notation we identify $\operatorname{Coh} (\spec \widehat{B})$ with finitely generated $\widehat{B}$-modules).  Now given coherent sheaves $F,F'$ on $\spec \widehat{B}$, the canonical map
\[ \operatorname{Ext}_{\spec \widehat{B}}^n( F,F') \stackrel{\sim}{\to} \operatorname{Ext}_{\operatorname{Spf} \widehat{B}}^n(i^\ast F, i^\ast F')\]
is an isomorphism for all $n\geq 0$ by \cite[III, 4.5.2]{ref:ega}. Therefore, the functor
\begin{equation} \label{eq:weakformalgaga}  i^\ast \colon D^-_{\operatorname{coh}}(\spec \widehat{B}) \stackrel{\sim}{\to} D^-_{\operatorname{coh}}(\operatorname{Spf} \widehat{B})\end{equation}
is an equivalence of categories  following the method of proof of Theorem \ref{thm:gagapc}.
 
\begin{theorem}[Affine formal GAGA for perfect complexes] \label{thm:weakformalgaga}
The equivalence of categories (\ref{eq:weakformalgaga}) restricts to an equivalence of categories
\[ i^\ast \colon \operatorname{Perf}( \spec \widehat{B}) \stackrel{\sim}{\to} \operatorname{Perf}(\operatorname{Spf} \widehat{B}).\]
\end{theorem}

The theorem above will follow from Corollary \ref{cor:zarpair} below, which morally asserts that the property of being perfect can be detected by base change to the residue field. Note it is important that we work with perfect complexes rather than just vector bundles, for otherwise the result is false. Indeed, over a Noetherian local ring $R$, the (ordinary) base change of any finite $R$-module $M$ to the residue field is finite free, but obviously not every finite module is projective.
 Corollary \ref{cor:zarpair} follows from \cite[\href{https://stacks.math.columbia.edu/tag/068V}{Tag 068V}]{ref:stacks} in the non-Noetherian case. Nonetheless, since we are in the Noetherian situation, we feel compelled to give a more intuitive proof, in the sense that we use standard d\'{e}vissage techniques.

\begin{lemma}
Let $(R,\mathfrak{m})$ be a Noetherian local ring and $P$ an object of $D(R)$. If the Tor amplitude of $P \otimes^{\operatorname{L}}_R R/\mathfrak{m}$ (as an object of $D(R/\mathfrak{m})$) lies in $[a,b]$, then for any $R$-module $N$ of finite length, 
\[ H^i(P \otimes_R^{\operatorname{L}} N) = 0\] 
for all $i \notin [a,b]$.
\end{lemma}

\begin{proof}
We will induct on the length of $N$, denoted $\ell(N)$. When $\ell(N) = 1$, necessarily $N \cong R/\mathfrak{m}$ and the result is clear. When $\ell(N) > 1$, we can find a submodule $N_0 \subseteq N$ such that $\ell(N_0) < \ell(N)$ and $\ell(N/N_0) = 1$. The result follows by considering the cohomology of the exact triangle
\[ P \otimes_R^{\operatorname{L}} N_0 \to P \otimes_R^{\operatorname{L}} N \to P \otimes_R^{\operatorname{L}} N/N_0. \] \end{proof}

\begin{proposition} \label{lem:perfres}
Let $(R, \mathfrak{m})$ be a complete Noetherian local ring and $P$ a pseudo-coherent object in $D(R)$. Suppose that $P \otimes_R^{\operatorname{L}} R/\mathfrak{m}$ has Tor amplitude in $[a,b]$. Then $P$ has Tor amplitude in $[a,b]$.

\end{proposition}

\begin{proof}
Choose a quasi-isomorphism $F^{\bullet} \to P$ where  $F^{\bullet}$ is a bounded above complex of finite free modules (such a $F^{\bullet}$ exists because $P$ is pseudo-coherent). The complex $F^\bullet$ is $K$-flat (in the sense of \cite{ref:spaltenstein}) and therefore can be used to compute derived tensor products. In other words, for any $R$-module $M$,
\[ P \otimes^{\operatorname{L}}_R M = F^{\bullet} \otimes_R M\]
in $D(R)$. We must now show that $H^i( F^{\bullet} \otimes_R M) = 0$ for all $i \notin [a, b]$. The formation of the (underived) tensor product commutes with direct limits and similarly for cohomology. Therefore, we may suppose that $M$ is finite, a fortiori complete and so
\[ F^{\bullet} \otimes_R M = F^{\bullet} \otimes_R \lim (M/\mathfrak{m}^kM) = \lim (F^\bullet \otimes_R M/\mathfrak{m}^kM ).\] 
On the other hand, by \cite[\href{https://stacks.math.columbia.edu/tag/091D}{Tag 091D}]{ref:stacks} we have
\[ R\lim (F^{\bullet} \otimes_R M/\mathfrak{m}^kM) = \lim (F^\bullet \otimes_R M/\mathfrak{m}^kM )\]
and so we can consider the Milnor exact sequence
\[ 0 \to R^1\lim H^{i-1}( F^{\bullet} \otimes_R M/\mathfrak{m}^kM) \to H^i (\lim F^{\bullet} \otimes_R M/\mathfrak{m}^kM) \to \lim H^i(F^{\bullet} \otimes_R M/\mathfrak{m}^kM) \to 0. \]
For every $k$, $M/\mathfrak{m}^kM$ is a finite length $R$-module and therefore the previous lemma implies the right term vanishes for all $i \notin [a,b]$ \textit{independent of $k$ and the module $M$}. The same is true of the left term for $i \notin [a+1, b+1]$. 

\medskip

To finish the proof, we have to show that  \[ R^1\lim H^{b}( F^{\bullet} \otimes_R M/\mathfrak{m}^kM)  = 0.\]
 There is a short exact sequence of complexes
\[  0 \to F^{\bullet} \otimes_R \mathfrak{m}^{k-1}M/\mathfrak{m}^{k}M \to  F^{\bullet} \otimes_R M/\mathfrak{m}^kM \to F^{\bullet} \otimes_R M/\mathfrak{m}^{k-1}M \to 0\]
which upon taking cohomology gives an exact sequence
\[\begin{tikzcd} H^b(F^{\bullet} \otimes_R M/\mathfrak{m}^kM) \ar{r} &  H^b(F^{\bullet} \otimes_R M/\mathfrak{m}^{k-1}M) \ar{r} & H^{b+1}( F^{\bullet} \otimes_R \mathfrak{m}^{k-1}M/\mathfrak{m}^{k}M ).  \end{tikzcd}\]
The right term is zero by the previous lemma since $\mathfrak{m}^{k-1}M/\mathfrak{m}^{k}M$ is an $R$-module of finite length. We have shown the inverse system $\{ H^b(F^{\bullet} \otimes_R M/\mathfrak{m}^kM)\}_k$ has all transition maps surjective and therefore 
\[R^1\lim H^{b}( F^{\bullet} \otimes_R M/\mathfrak{m}^kM)  = 0.\]
\end{proof}

\begin{corollary} \label{cor:zarpair}
Let $(R, I)$ be a Zariski pair, i.e., a Noetherian ring $R$ and an ideal $I$ contained in the Jacobson radical of $R$. Let $P$ a pseudo-coherent object of $D(R)$. If $P \otimes_R^{\operatorname{L}} R/I$ has Tor amplitude in $[a,b]$ then $P$ has Tor amplitude in $[a,b]$.
\end{corollary}

\begin{proof}

The assumption on the Tor amplitude of $P \otimes_R^{\operatorname{L}} R/I$ implies for any maximal ideal $\mathfrak{m}$ of $R$, the Tor amplitude of $(P \otimes_R^{\operatorname{L}} R/I)_{\mathfrak{m}} = P_{\mathfrak{m}} \otimes_{R_\mathfrak{m}}^{\operatorname{L}} R_\mathfrak{m}/I_\mathfrak{m}$ is in $[a,b]$. Therefore, the same is true of
\[ ( P_{\mathfrak{m}} \otimes_{R_{\mathfrak{m}}}^{\operatorname{L}} R_{\mathfrak{m}}/I_{\mathfrak{m}}  ) \otimes_{R_{\mathfrak{m}}/I_{\mathfrak{m}}  }^{\operatorname{L}}   R_{\mathfrak{m}}/\mathfrak{m}R_{\mathfrak{m}} = P_{\mathfrak{m}} \otimes_{R_{\mathfrak{m}}}^{\operatorname{L}} R_{\mathfrak{m}}/\mathfrak{m}R_{\mathfrak{m}} \]
and hence also of
\[  P_{\mathfrak{m}} \otimes_{R_{\mathfrak{m}}}^{\operatorname{L}}   \widehat{R}_{\mathfrak{m}}/\mathfrak{m} \widehat{R}_{\mathfrak{m}} = (P_{\mathfrak{m}} \otimes_{R_{\mathfrak{m}}}^{\operatorname{L}} \widehat{R}_{\mathfrak{m}} ) \otimes^{\operatorname{L}}_{\widehat{R}_{\mathfrak{m}}}  \widehat{R}_{\mathfrak{m}}/\mathfrak{m} \widehat{R}_{\mathfrak{m}} , \]
where $\widehat{R}_\mathfrak{m}$ is the $\mathfrak{m}$-adic completion of $R_\mathfrak{m}$. By  Proposition \ref{lem:perfres} we conclude that $P_{\mathfrak{m}} \otimes_{R_{\mathfrak{m}}}^{\operatorname{L}} \widehat{R}_{\mathfrak{m}} $ has Tor amplitude in $[a,b]$. By faithful flatness, $P_\mathfrak{m}$ has Tor amplitude in $[a,b]$. Since $\mathfrak{m}$ was any maximal ideal of $R$, we are done.
\end{proof}

\begin{proof}[Proof of Theorem \ref{thm:weakformalgaga}]
It is clear that
\[i^\ast \colon \operatorname{Perf}( \spec \widehat{B}) \to \operatorname{Perf}(\operatorname{Spf} \widehat{B})\]
is fully faithful. To prove essential surjectivity, by equivalence (\ref{eq:weakformalgaga}) it is enough to show the following. Let $P$ be an object of $D^-_{\operatorname{coh}}( \spec \widehat{B})$. If $i^\ast P$ is perfect, then $P$ is perfect. To this end, consider the $2$-commutative diagram
\[ \begin{tikzcd} D^{-}_{\operatorname{coh}}( \spec \widehat{B}) \ar{r} \ar{d} & D^{-}_{  \operatorname{coh}}( \operatorname{Spf} \widehat{B}) \ar{d} \\
D^{-}_{\operatorname{coh}} ( \spec \widehat{B}/\widehat{J}) \ar{r} &  D^-_{\operatorname{coh}}( \operatorname{Spf} \widehat{B}/\widehat{J}).\end{tikzcd} \]
Notice that $\spec \widehat{B}/\widehat{J}$ and $\operatorname{Spf} \widehat{B}/\widehat{J}$ are \textit{canonically isomorphic} as ringed spaces. Therefore, it is enough to show that if $P$ is a pseudo-coherent object of $D(\spec \widehat{B})$ such that $P \otimes_{\widehat{B}}^{\operatorname{L}} \widehat{B}/\widehat{J}$ is perfect, then $P$ is perfect. Since  $(\widehat{B}, \widehat{J})$ is a Zariski pair, Corollary \ref{cor:zarpair} applies and we win.
\end{proof}

Finally, we record a consequence of Theorem \ref{thm:weakformalgaga} that will be used in the next subsection. Keeping with the notation above, let $B$ be a Noetherian ring, $J \subseteq B$ an ideal and $\widehat{B}$ the $J$-adic completion of $B$. There are canonical flat morphisms of ringed spaces $i \colon \operatorname{Spf} \widehat{B}  \to  \spec \widehat{B}$ and $ j \colon \spec \widehat{B} \to  \spec B$
 which induce pullback maps
\[ i^\ast \colon D^-_{\operatorname{coh}}(\spec \widehat{B}) \to D^-_{\operatorname{coh}}(\operatorname{Spf} \widehat{B}) \hspace{6mm} \text{and} \hspace{6mm} j^\ast \colon D^-_{\operatorname{coh}}(\spec B) \to D^-_{\operatorname{coh}}(\spec \widehat{B}). \]

\begin{corollary}  \label{prop:corweakformalgaga}
Let $P$ be an object of $D^-_{\operatorname{coh}}(\spec B)$ and define $\iota$ to be the composition
\[ \iota \equiv j \circ i \colon \operatorname{Spf} \widehat{B} \to\spec \widehat{B} \to  \spec B.\]
Then $\iota^\ast P$ is perfect if and only if $ j^\ast P$ is perfect.
\end{corollary}

\begin{proof}
The non-trivial direction to prove is $\iota^\ast P$ perfect implies that $j^\ast P$ is perfect. However, $ \iota^\ast P =  i^\ast (  j^\ast P)$ and therefore the result follows immediately from the proof of Theorem \ref{thm:weakformalgaga} above.
\end{proof}

\subsection{Formal GAGA for Perfect Complexes}

\begin{theorem} \label{thm:gagaperf} Let $\mathscr{X} \to \spec A$ be an algebraic  stack that is proper over an $I$-adically complete Noetherian ring $A$. The equivalence
\[ \iota^\ast \colon D^-_{\operatorname{coh}}(\mathscr{X}) \stackrel{\sim}{\to} D^-_{\operatorname{coh}}(\widehat{\hspace{0pt}\mathscr{X}})\]
of Theorem \ref{thm:gagapc} restricts to an equivalence
\[ \iota^\ast \colon \operatorname{Perf}(\mathscr{X}) \stackrel{\sim}{\to} \operatorname{Perf}( \widehat{\hspace{0pt}\mathscr{X}} ).\]
\end{theorem}

\begin{proof}
It is clear that the restriction $\iota^\ast \colon \operatorname{Perf}(\mathscr{X}) \to \operatorname{Perf}( \widehat{\hspace{0pt}\mathscr{X}} )$ is fully faithful. To show essential surjectivity, by Theorem \ref{thm:gagapc}  this amounts to showing that if an object $P$ of $D^-_{\coh}(\mathscr{X})$ maps into $\operatorname{Perf}(\widehat{\hspace{0pt}\mathscr{X}})$, then $P$ is perfect. Without loss of generality, we may assume that the pullback of $I$ to $\mathscr{X}$ is not the unit ideal, for otherwise $\mathscr{X} = \widehat{\hspace{0pt}\mathscr{X}}$ and there is nothing to prove. 
\medskip

Let  $U = \spec B \to \mathscr{X}$ be a smooth surjection from an affine scheme. First we will show for any maximal ideal $\mathfrak{m}$ containing $IB$ that the restriction of $P$ to $\spec B_{\mathfrak{m}}$ \textit{in the Zariski topology} is perfect (The reader that prefers a more deformation-theoretic viewpoint may interpret this as saying for any closed point $x_0$ in $\mathscr{X}$, the restriction of $P$ to a versal ring of $\mathscr{X}$ at $x_0$ is perfect). 
The object $ \iota^\ast P$ is perfect on $\widehat{\hspace{0pt}\mathscr{X}}$ and therefore $ \iota^\ast( P|_{\spec B_{\mathfrak{m},\liset}} )$ is a perfect  object of $D^-_{\operatorname{coh}}(\widehat{\spec B}_{\mathfrak{m}, {\liset}})$. 
On the other hand, the equivalence 
\[  (\widehat{\epsilon} \circ \widehat{\delta})^\ast  \colon D^-_{\operatorname{coh}}(\widehat{\spec B}_{\mathfrak{m}, {\zar}}) \stackrel{\sim}{\to} D^-_{\operatorname{coh}}(\widehat{\spec B}_{\mathfrak{m}, {\liset}})\]
given by Proposition \ref{prop:pctopapp} restricts to the full subcategories of perfect complexes by Proposition \ref{prop:tordimapp}. Therefore,  we may work with the Zariski topology and must show that if $Q \in D^-_{\coh}(\spec B_{\mathfrak{m}})$ satisfies the property that $\iota^\ast Q$ is perfect, then $Q$ is perfect.

\medskip

By Corollary \ref{prop:corweakformalgaga}, we deduce that  $ j^\ast Q$ is perfect where (keeping with the notation in the previous subsection)   $ j^\ast$ is the pullback
\[ j^\ast \colon D^-_{\coh}(\spec \widehat{B_{\mathfrak{m}}}) \to  D^-_{\coh}(\spec B_{\mathfrak{m}}).\]
Now the ideal $IB$ is contained in the maximal ideal  $\mathfrak{m}$ and therefore the map $B_{\mathfrak{m}}$ to the $I$-adic completion $\widehat{B_{\mathfrak{m}}}$ is faithfully flat. We conclude that $Q$ is perfect  by \cite[\href{https://stacks.math.columbia.edu/tag/068T}{Tag 068T}]{ref:stacks}. This completes the claim that $P|_{\spec B_{\mathfrak{m}}}$ is perfect in the Zariski topology.
\medskip

The preceding argument shows that the restriction of $P$ to the (Zariski) local ring at any closed point of $U_0\coloneqq U \times_A A/I$ is perfect. By spreading out, there is a Zariski open $V$ containing $U_0$ such that $P|_V$ is perfect. Therefore to show that $P$ is perfect, it is enough to show that the  composition $V \to U \to \mathscr{X}$ is surjective. Let $W$ denote the image of this morphism. It is an open substack of $\mathscr{X}$ containing $\mathscr{X}_0$. If the complement $|\mathscr{X}| -  |W|$ is non-empty, by properness its image in $\spec A$ is a non-empty closed set disjoint from $\spec A/I$, contradicting the adic property of $A$. Therefore $|W| =  |\mathscr{X}|$ and so $V \to \mathscr{X}$ is surjective by \cite[\href{https://stacks.math.columbia.edu/tag/04XI}{Tag 04XI}]{ref:stacks}.
\end{proof}

\subsection{Grothendieck's Existence Theorem for Perfect Complexes}

\begin{theorem}[Grothendieck Existence]
Let $\mathscr{X}$ be an algebraic stack that is proper  over an $I$-adically complete Noetherian ring $A$. Define $\mathscr{X}_n\coloneqq \mathscr{X} \times_A A/I^{n+1}$. Let $\{P_n\}$ be an adic system of perfect complexes on $\mathscr{X}_n$. Then there exists a perfect complex $P \in \operatorname{Perf}(\mathscr{X})$ such that $L(j_{n,\qc})^\ast P \cong P_n$.
\end{theorem}

\begin{proof}
By formal GAGA for perfect complexes (Theorem \ref{thm:gagaperf}) and arguing as in Theorem \ref{thm:existencepc}, it is enough to show the following. Let $Q_k$ be an adic system of perfect complexes with each $Q_k \in \operatorname{Perf}( \mathscr{X}, j_{k, \ast} \mathscr{O}_{\mathscr{X}_k})$. Then \[Q\coloneqq R\lim Q_k\]
 is perfect on $\widehat{\mathscr{X}}$. 
\medskip

As the question of being perfect is local, just as in the proof of Theorem \ref{thm:existencepc} we may reduce to the situation that $\mathscr{X} = \spec B$ and $Q_k$ are perfect on $(\spec B_k,  \mathscr{O}_{\spec B_k})$ in the Zariski topology. Furthermore, by the proof of Theorem \ref{thm:existencepc}, we know that $Q\coloneqq R\lim Q_k$ is an object of $D^-_{\coh}(\spec B, \widehat{\mathscr{O}}_{\spec B}) \simeq D^-_{\coh}(\operatorname{Spf} \widehat{B})$ such that its pullback to $\spec B_0$ is isomorphic to $Q_0$. Now recall by (\ref{eq:weakformalgaga}) that there is an equivalence of categories
\[ \iota^\ast\colon D^-_{\coh}(\spec \widehat{B}) \to D^-_{\coh}(\operatorname{Spf} \widehat{B}).\]
It follows that $Q$ is the image of a pseudo-coherent object $Q'$ on $\spec \widehat{B}$ with the property that $Q' \otimes_{\widehat{B}}^{\operatorname{L}} \widehat{B}/I \widehat{B}$ is perfect. By Corollary \ref{cor:zarpair}, $Q'$ is perfect and therefore $Q$ is perfect.
\end{proof}

\section{Grothendieck's Existence Theorem for Relatively Perfect Complexes} \label{section:existencerp}

\begin{definition} \label{def:rp} Let $S$ be an algebraic space and $f \colon \mathscr{X} \to S$ an algebraic stack over $S$ that is flat and locally of finite presentation. An object $P \in D({\mathscr{X}}_{\liset})$ is said to be \textit{relatively perfect} over $S$ if it is pseudo-coherent as an object of $D({\mathscr{X}}_{\liset})$ and locally has finite Tor dimension as an object of $D(f^{-1}\mathscr{O}_{S_{\liset}})$.
\end{definition}

\begin{remark}
Let $X \to \spec k$ be the nodal cubic curve. Then the structure sheaf of the node is relatively perfect over $\spec k$, but is \textit{not} perfect on $X$. This shows that the notion of being relatively perfect does not mean perfect on fibers.
\end{remark}

\begin{proposition}
Let $\mathscr{X}$ be an algebraic stack that is proper and flat over an $I$-adically complete Noetherian ring $A$. Define $\mathscr{X}_n \coloneqq \mathscr{X} \times_A A/I^{n+1}$ and let $j_n \colon \mathscr{X}_n \hookrightarrow \mathscr{X}$ denote the canonical closed immersion of stacks. Let $P_n \in D(\mathscr{X}_n)$ be an adic system of relatively perfect objects over $\spec A/I^{n+1}$. Then there is a object $P \in D(\mathscr{X})$ that is relatively perfect over $\spec A$ such that $L(j_{n,\qc})^\ast P = P_n$.
\end{proposition}

\begin{proof}

Let $\spec B\to \mathscr{X}$ be a smooth cover by an affine scheme. The restriction of $L(j_{n,qc})^\ast$ to the (small) \'{e}tale site of $\spec B$ agrees with the \'{e}tale pullback $L(j_{n,\et})^\ast$.  Therefore, as the question of being relatively perfect over $A$ is local on $\mathscr{X}$, we may assume that $\mathscr{X} = \spec B$ with the \'{e}tale topology. Furthermore, just as in the proof of Theorem  \ref{thm:gagaperf}, by 
\cite[\href{https://stacks.math.columbia.edu/tag/0DK7}{Tag 0DK7}]{ref:stacks} and  \cite[\href{https://stacks.math.columbia.edu/tag/0DHY}{Tag 0DHY}]{ref:stacks}  we may work in the Zariski topology and prove the following. Let $\spec B \to \spec A$ be a flat morphism of finite type, with $A$ an $I$-adically complete Noetherian ring $A$. Let $P$ be a pseudo-coherent object of $D(\spec B)$ and suppose that $P \otimes_B^{\operatorname{L}} B/IB$ has finite Tor dimension as an $A/I$-module. Then $P$ has finite Tor dimension as an $A$-module. To prove this, observe by flatness that the natural map
\[ B \otimes_A^{\operatorname{L}} A/I \to B/IB\]
is a quasi-isomorphism. Therefore,
\begin{eqnarray*} P \otimes_B^{\operatorname{L}} B/IB &=&  P \otimes_B^{\operatorname{L}} (B \otimes^{\operatorname{L}}_A A/I) \\
&=&  P  \otimes^{\operatorname{L}}_A A/I.\end{eqnarray*}
 Since $A \to B$ is finite type, $P$ is a pseudo-coherent object of $D(A)$. The result follows from Corollary \ref{cor:zarpair}.
\end{proof}

\appendix
\section{Comparison between Topologies}

\subsection{Comparisons between topologies on a scheme} 

Let $X$ be a locally Noetherian scheme and $X_0 \subseteq X$ a closed subscheme defined by a coherent ideal $I$. Previously, recall that we defined the ringed site 
\[ \widehat{X} \coloneqq (X_{\liset},\widehat{ \mathscr{O}}_{ {X}_{\liset}}).\] Let  ${X}_{\zar}$ and ${X}_{\et}$ denote respectively the Zariski and (small) \'{e}tale sites of $X$. Analogously, we may define $\widehat{X}_{\zar}$ and $\widehat{X}_{\et}$. In this appendix, we will record  comparison results between the the categories of pseudo-coherent objects on $\widehat{X}_{\zar}, \widehat{X}_{\et}$ and $\widehat{X}$. Similarly, we will also record a result comparing Tor dimensions of pseudo-coherent objects on these three sites. We let $\we \colon \widehat{X}_{\et} \to \xhzar$ and $\widehat{\delta} \colon \widehat{X}_{\liset} \to \widehat{X}_{\et}$ the canonical flat morphisms of ringed sites.

\begin{proposition} \label{prop:gagatopology}
The pullback maps $\we^\ast$ and $\widehat{\delta}^\ast$ induce equivalences of categories
\[ \operatorname{Coh}(\widehat{X}_{\zar}) \stackrel{ \we^\ast}{\to}    \operatorname{Coh}(\widehat{X}_{\et}) \stackrel{ \widehat{\delta}^\ast}{\to}  \operatorname{Coh}( \widehat{X}_{\liset})  .\]
\end{proposition}

\begin{proof}
See \cite[Remark 1.6]{ref:conrad}.
\end{proof}
 
\begin{lemma} \label{lem:vanishinget}
Let $X$ be an affine Noetherian scheme, $I$ a coherent ideal on $X$. For any coherent sheaf $G$ on $\widehat{X}_{\et}$, $H^i(\widehat{X}_{\et},G)=0$ for all $i > 0$. 
\end{lemma}

\begin{proof}
Define $G_n \coloneqq G/I^{n+1}G$. First we prove that the natural map $G \to R\lim G_n$ is an isomorphism. Since $G \cong \lim G_n$, it is enough to show $H^i(R \lim G_n) = 0$ for all $i > 0$. Now $H^i(R\lim G_n)$ is the sheafification of the presheaf $\widehat{U} \mapsto H^i(\widehat{U}, R\lim G_n)$ for $\widehat{U} \in \operatorname{Ob}(\widehat{X}_{\et})$. Therefore we reduce to showing for $Y = \spec A$ that 
\[ H^i(\widehat{Y}_{\et}, R\lim G_n) = 0\] for all $i > 0$. Define $Y_n \coloneqq \spec A/I^{n+1}$. The Milnor exact sequence  gives
\begin{equation} \label{eq:milnorseq} 0 \to R^1 \lim H^{i-1} (\widehat{Y}_{\et}, G_n) \to H^i(\widehat{Y}_{\et}, R\lim G_n) \to \lim H^i(\widehat{Y}_{\et}, G_n) \to 0. \end{equation}
We have
\begin{eqnarray*} H^i(\widehat{Y}_{\et}, G_n) &=& H^i((Y_n)_{\et}, G_n) \\
&=& H^i((Y_{n})_{\zar}, G_n)
\end{eqnarray*}
by  \cite[\href{https://stacks.math.columbia.edu/tag/03DW}{Tag 03DW}]{ref:stacks}. Therefore by (\ref{eq:milnorseq}), 
\[ H^i(\widehat{Y}_{\et}, R\lim G_n) = 0\]
for all $i \geq 2$. To show vanishing when $i = 1$, we must show that
\begin{equation} R^1 \lim H^{0} (\widehat{Y}_{\et}, G_n)  = 0. \label{eq:appsurj} \end{equation}
Now  $Y$ is affine and therefore $H^0(\widehat{Y}_{\et}, G_n) \to H^0(\widehat{Y}_{\et}, G_{n-1})$ is surjective. Hence (\ref{eq:appsurj}) follows and so $G \stackrel{\sim}{\to} R\lim G_n.$ Repeating exactly the same argument, we get that
\[ H^i(\widehat{X}_{\et},G) = 0\]
for all $i > 0$ as desired. \end{proof}

\begin{corollary} \label{cor:vanishinget}
Let $X$ be a locally Noetherian scheme and $G$ a coherent sheaf on $\widehat{X}_{\et}$. Then $R^i\we_\ast G = 0$ for all $i > 0$.
\end{corollary}

\begin{corollary} \label{cor:gagatop}
Let $X$ be a locally Noetherian scheme and $F$ a coherent sheaf on $\widehat{X}_{\zar}$. Then 
\[ H^i(\widehat{X}_{\zar}, F) = H^i(\widehat{X}_{\et}, \we^\ast F)\]
for all $n \geq 0.$
\end{corollary}

\begin{proof}
We have $H^i(\widehat{X}_{\et}, \we^\ast F) = H^i(\widehat{X}_{\zar}, R\we_\ast \we^\ast F)$. By the coherence of $\we^\ast F$,  Corollary \ref{cor:vanishinget} and the full-faithfulness of $\we^\ast$ (Proposition \ref{prop:gagatopology}) we have $R\we_\ast \we^\ast F = \we_\ast \we^\ast F = F$. The result follows.
\end{proof}

\begin{proposition} \label{prop:pctopapp}
We have natural equivalences of categories
\[      D^-_{\operatorname{coh}}(\widehat{X}_{\zar}) \stackrel{ \we^\ast}{\to} D^-_{\operatorname{coh}}(\widehat{X}_{\et})  \stackrel{\delta^\ast}{\to} D^-_{\operatorname{coh}}(\widehat{X}_{\liset}).\]
\end{proposition}

\begin{proof}
We first prove that $\we^\ast$ is an equivalence using Theorem \ref{thm:genpc}. Condition (\ref{heart}) is Proposition \ref{prop:gagatopology} above. Condition (\ref{vanishing}) is Lemma \ref{lem:vanishinget} above and condition (\ref{objects}) is true because an open immersion is an \'{e}tale morphism. It remains to check Condition (\ref{ext}). Choose $F,F'\in \operatorname{Coh}(\xhzar)$. Then following \cite[0\textsubscript{III}, 12.3.3, 12.3.4]{ref:ega}, one proves that
\[ \we^\ast \uext^n_{\xhzar}(F,F') \stackrel{\sim}{\to} \uext^n_{\widehat{X}_{\et}}(\we^\ast F, \we^\ast F')\]
is an isomorphism. Now argue using this isomorphism, Corollary \ref{cor:gagatop} and the local-to-global spectral sequence as in \cite[III\textsubscript{1}, 4.5.2]{ref:ega} that 
\[ \operatorname{Ext}_{\xhzar}^n(\mathscr{F},\mathscr{G}) \stackrel{\sim}{\to} \operatorname{Ext}_{\widehat{X}_{\et}}^n(\we^\ast \mathscr{F}, \we^\ast\mathscr{G}).\]
All the conditions of Theorem \ref{thm:genpc} are satisfied and therefore \[ \we^\ast \colon D^-_{\operatorname{coh}}(\xhzar) \to D^-_{\operatorname{coh}}(\widehat{X}_{\et}) \] is an equivalence. Finally, the equivalence
\begin{equation} \widehat{\delta}^\ast \colon D^-_{\operatorname{coh}}(\widehat{X}_{\et})  \stackrel{\sim}{\to} D^-_{\operatorname{coh}}(\widehat{X}_{\liset}) \label{eq:hd} \end{equation}
 is similarly shown. We leave this to the reader; the key point being \cite[Proposition 12.7.4]{ref:lmb} which states that $\widehat{\delta}$ induces an equivalence of ringed topoi
\[ (\operatorname{Sh}(X_{\et}), \widehat{\mathscr{O}}_{X_{\et}}) \stackrel{\sim}{\to}  (\operatorname{Sh}(X_{\liset}), \widehat{\mathscr{O}}_{X_{\liset}}).\]
\end{proof}

\begin{corollary} \label{cor:appzaret}
Let $P_n$ be an inverse system of objects in $D^-_{\operatorname{coh}}(\xhzar)$ such that $R\lim P_n \in D^-_{\operatorname{coh}}(\xhzar)$. Then the canonical morphism 
\[ \we^\ast R\lim P_n \to R\lim \we^\ast P_n\]
is an isomorphism in $D(\widehat{X}_{\et})$.
\end{corollary}

\begin{proof}
Let $\widehat{U}_{\et}$ be any object of $\widehat{X}_{\et}$. Then we have
\begin{eqnarray*} R\Gamma(\widehat{U}_{\et}, \we^\ast R\lim P_n) &=&R \operatorname{Hom}_{\mathscr{O}_{\widehat{U}_{\et}}}(\mathscr{O}_{\widehat{U}_{\et}}, \we^\ast R\lim P_n) \\
&=& R\Hom_{\mathscr{O}_{\widehat{U}_{\et}}}(\we^\ast \mathscr{O}_{\widehat{U}_{\zar}}, \we^\ast R\lim P_n) \\
&=&  R\Hom_{\mathscr{O}_{\widehat{U}_{\zar}}}(\mathscr{O}_{\widehat{U}_{\zar}}, R\lim P_n)  \hspace{1.4cm} \text{(Proposition \ref{prop:pctopapp}) }\\
&=& R\lim  R\Hom_{\mathscr{O}_{\widehat{U}_{\zar}}}(\mathscr{O}_{\widehat{U}_{\zar}}, P_n) \\
&=& R\lim  R\Hom_{\mathscr{O}_{\widehat{U}_{\et}}}  (\we^\ast \mathscr{O}_{\widehat{U}_{\zar}}, \we^\ast P_n)  \hspace{1cm} \text{(Proposition \ref{prop:pctopapp})} \\
&=& R\Hom_{\mathscr{O}_{\widehat{U}_{\et}}}(\we^\ast \mathscr{O}_{\widehat{U}_{\zar}},R\lim \we^\ast P_n) \\
&=& R\Hom_{\mathscr{O}_{\widehat{U}_{\et}}}(\mathscr{O}_{\widehat{U}_{\et}},R\lim \we^\ast P_n) \\
&=& R\Gamma(\widehat{U}_{\et}, R\lim \we^\ast P_n).
\end{eqnarray*}
We used the fact that $R\lim P_n \in D^-_{\operatorname{coh} }(\xhzar)$ to pass from the second to third equality. Since $\widehat{U}_{\et}$ was arbitary, we are done.
\end{proof}

\begin{remark} \label{cor:applisetet}
Let $P_n$ be an inverse system of objects in $D^-_{\operatorname{coh}}(\widehat{X}_{\et})$. Consider the morphism of sites $\widehat{\delta} \colon \widehat{X}_{\liset} \to \widehat{X}_{\et}$. This morphsim is cocontinuous by  \cite[\href{https://stacks.math.columbia.edu/tag/0788}{Tag 0788}]{ref:stacks}. It follows the canonical morphism $\widehat{\delta}^\ast R\lim P_n \to R\lim \widehat{\delta}^\ast P_n$ is an isomorphism $D(\widehat{X}_{\liset})$.
\end{remark}

\begin{proposition} \label{prop:tordimapp}
Let $X$ be a locally Noetherian scheme and let  $\we \colon  \widehat{X}_{\et} \to \xhzar$, $ \hd \colon  \widehat{X}_{\liset}\to \widehat{X}_{\et}$ denote the canonical morphisms of sites. An object $P \in D(\widehat{X}_{\et})$, has finite Tor dimension if and only if the same is true of $\we^\ast P$. Similarly, $Q \in D(\widehat{X}_{\liset})$ has finite Tor dimension if and only if the same is true of $\hd^\ast Q$.
\end{proposition}

\begin{proof}
Follows analogously as in \cite[\href{https://stacks.math.columbia.edu/tag/08HF}{Tag 08HF}]{ref:stacks}. \end{proof}

\section{Derived Tensor Product and Sheafification}
Let $\mathscr{A}$ be an abelian category. Recall that a complex $I^{\bullet}$ is said to be \textit{$K$-injective} if for any acyclic complex $A^{\bullet}$, 
\[ \operatorname{Hom}_{K(\mathscr{A})}(A^{\bullet}, I^{\bullet}) = 0,\]
where $K(\mathscr{A})$ is the homotopy category. Recall also that a complex $F^{\bullet}$ is said to be \textit{$K$-flat} if for any acyclic complex $A^{\bullet}$, the complex
\[ \operatorname{Tot}(A^{\bullet} \otimes F^{\bullet})\]
is acyclic. It is a fact that if $\mathscr{A} = \operatorname{Mod}(\mathscr{O}_X)$, the category of sheaves of modules on a ringed site $(X,\mathscr{O}_X)$, then every complex has a quasi-isomorphism to a $K$-injective one \cite[\href{https://stacks.math.columbia.edu/tag/01DU}{Tag 01DU}]{ref:stacks}, and a quasi-isomorphism from a $K$-flat one \cite[\href{https://stacks.math.columbia.edu/tag/06YS}{Tag 06YS}]{ref:stacks}. The same is true for the category of presheaves, $\operatorname{PMod}(\mathscr{O}_X)$.
\medskip

There is an adjoint pair
\[ (-)^{\#} \colon \operatorname{PMod}(\mathscr{O}_X) \rightleftarrows\operatorname{Mod}(\mathscr{O}_X) \colon \fgt\]
where $\fgt$ is the forgetful functor and $(-)^{\#}$ denotes sheafification. The forgetful functor is left exact, while sheafification is exact. Since $\operatorname{Mod}(\mathscr{O}_X)$ is  a Grothendieck abelian category, by \cite[Corollary 3.14]{ref:serpe} this extends to an adjoint pair 
\[ (-)^{\#} \colon D(\operatorname{PMod}(\mathscr{O}_X)) \rightleftarrows D(\operatorname{Mod}(\mathscr{O}_X)) \colon R\fgt.\]

\begin{lemma} \label{lem:apptricky}
Let $(X,\mathscr{O}_X)$ be a ringed site and $F^{\bullet} \to J^{\bullet}$ a quasi-isomorphism of complexes of $\mathscr{O}_X$-modules. Then for any $K$-injective complex $I^{\bullet}$, the canonical morphism
\[ \uhom^{\bullet}(\fgt(J^{\bullet}), \fgt(I^{\bullet})) \to   \uhom^{\bullet}(\fgt(F^{\bullet}), \fgt(I^{\bullet}))\]
is a quasi-isomorphism in the homotopy category $K(\operatorname{PMod}(\mathscr{O}_X))$.
\end{lemma}

\begin{proof}
We begin by recalling the following fact. Any additive functor between abelian categories $\mathscr{A} \to \mathscr{B}$ induces an exact functor $K(\mathscr{A}) \to K(\mathscr{B})$.
Now let $C^{\bullet}$ be the cone of $F^{\bullet} \to J^{\bullet}$. Applying $\fgt(-)$ and then $\uhom^{\bullet}( -, \fgt(I^{\bullet}))$, we get the following exact triangle in $K(\operatorname{PMod}(\mathscr{O}_X))$:
\[\uhom^{\bullet}(\fgt(C^{\bullet}), \fgt(I^{\bullet})) \to      \uhom^{\bullet}(\fgt(J^{\bullet}), \fgt(I^{\bullet})) \to \uhom^{\bullet}(\fgt(F^{\bullet}), \fgt(I^{\bullet})). \]
 To prove the lemma, we must show for every object $j\colon U \to X$ and integer $n \in \mathbf{Z}$ that 
\begin{equation} \label{eq:bzero} H^n( \Gamma(U, \uhom^\bullet (\fgt(C^{\bullet}),\fgt(I^{\bullet}))))= 0.\end{equation}
 We have
\begin{eqnarray*} H^n( \Gamma(U, \uhom^\bullet (\fgt(C^{\bullet}),\fgt(I^{\bullet})))) &=& \Hom_{K(\operatorname{PMod}(\mathscr{O}_U))}( \fgt(j^\ast C^{\bullet}), \fgt(j^\ast I^{\bullet})[n]) \\
&=&  \Hom_{K(\operatorname{PMod}(\mathscr{O}_U))}( \fgt(j^\ast C^{\bullet}), \fgt(j^\ast I^{\bullet}[n])) \\
&=& \Hom_{K(\operatorname{Mod}(\mathscr{O}_U))}( \fgt(j^\ast C^{\bullet})^{\#}, (j^\ast I^{\bullet}[n])) \\
&=&  \Hom_{K(\operatorname{Mod}(\mathscr{O}_U))}( j^\ast C^{\bullet}, j^\ast I^{\bullet}[n]) \\
&=&  \Hom_{K(\operatorname{Mod}(\mathscr{O}_U))}( j_!j^\ast C^{\bullet}, I^{\bullet}[n]) .
\end{eqnarray*}
We used adjunction between $\fgt$ and $(-)^{\#}$ in the third equality. In the fourth equality, we used the fact that forgetting and then sheafifiying is the same as the identity! Now $j_! j^\ast C^{\bullet}$ is acyclic since $j^\ast,j_!$ are exact and $C^\bullet$ is acyclic. It follows by definition of being $K$-injective that 
\[ \Hom_{K(\operatorname{Mod}(\mathscr{O}_U))}(j_! j^\ast C^{\bullet},  I^{\bullet}[n]) = 0.\]
Therefore (\ref{eq:bzero}) is zero and we win. \end{proof}

\begin{lemma}
Let $(X,\mathscr{O}_X)$ be a ringed site and $\mathscr{F},\mathscr{G}$ any two objects of $\dm$. Then: 
\[ R\fgt R\underline{\Hom}(\mathscr{F},\mathscr{G}) = R\uhom(R\fgt(\mathscr{F}), R\fgt (\mathscr{G}))\]
in $\dpm$.
\end{lemma}

\begin{proof}
Let $F^{\bullet} \to  \mathscr{F}$ be a $K$-flat resolution of $\mathscr{F}$ and $\mathscr{G} \to I^{\bullet}$ a  $K$-injective resolution of $\mathscr{G}$. Then the object $R\uhom(\mathscr{F},\mathscr{G}) $ can be represented in $\dm$ by the  complex $\uhom^{\bullet} (F^{\bullet}, I^{\bullet})$, with $n$-th term of the form
\[ \uhom^n (F^{\bullet}, I^{\bullet}) = \prod_{n=p+q} \uhom(F^{-q}, I^p)\]
with differential given by $d(f) = d_F\circ f - (-1)^n f \circ d_I.$ Furthermore, given any acyclic complex $A^{\bullet}$, we have
\begin{eqnarray*} \Hom_{K(\operatorname{Mod}(\mathscr{O}_X))}(A^{\bullet}, \uhom^{\bullet} (F^{\bullet}, I^{\bullet}))&=& \Hom_{K(\operatorname{Mod}(\mathscr{O}_X))} (\operatorname{Tot}(A^{\bullet} \otimes F^{\bullet}), I^{\bullet}) \\
&=& 0
\end{eqnarray*}
since $F^{\bullet}$ is $K$-flat and $I^{\bullet}$ is $K$-injective. The upshot is that $\uhom^{\bullet} (F^{\bullet}, I^{\bullet})$ is $K$-injective and therefore
\begin{equation} \label{eq:b1} R \fgt R\uhom(\mathscr{F},\mathscr{G})  = \fgt\left(\uhom^{\bullet} (F^{\bullet}, I^{\bullet})\right) \end{equation}
in $\dpm$.
\medskip

On the other hand, let $F^{\bullet} \to J^{\bullet}$ be a $K$-injective resolution of $F^{\bullet}$. By the previous lemma, we have
\[ \uhom^{\bullet}(\fgt(F^{\bullet}), \fgt(I^{\bullet})) = \uhom^{\bullet}(\fgt(J^{\bullet}), \fgt(I^{\bullet}))\]
in the \textit{derived category} $\dpm$.  Now recall that $\fgt$ is right adjoint to the exact functor $(-)^{\#}$ and thus commutes with arbitrary limits, hence
 \begin{equation} \label{eq:b2}   \fgt\left(\uhom^{\bullet} (F^{\bullet}, I^{\bullet})\right) =  \fgt\left(\uhom^{\bullet} (J^{\bullet}, I^{\bullet})\right).  \end{equation}
 By assumption that $I^{\bullet}, J^\bullet$ are $K$-injective, we have 
  \begin{eqnarray*} R\fgt(\mathscr{G}) &=& \fgt(I^{\bullet}), \\
   R\fgt(\mathscr{F}) &=& R\fgt(F^\bullet) =  \fgt(J^{\bullet}) \end{eqnarray*}
  in $\dpm$ \footnote{The careful reader will note it is important these equalities are in the derived category and \textit{not} the homotopy category.}. Furthermore, $R\fgt(I^{\bullet})$ is $K$-injective because $\fgt$ preserves $K$-injectives and therefore 
\begin{equation} \label{eq:b3} \uhom^{\bullet} (\fgt(J^{\bullet}),\fgt( I^{\bullet})) =  R\uhom(R\fgt(\mathscr{F}), R\fgt(\mathscr{G}))    \end{equation}
in $\dpm$. Now combine (\ref{eq:b1}), (\ref{eq:b2}) and (\ref{eq:b3}) to deduce the result.
\end{proof}

\begin{remark}
The reader may question the necessity of Lemma \ref{lem:apptricky} to conclude that (\ref{eq:b2}) is true. Indeed, it is tempting to argue that because $\uhom^{\bullet}(C^{\bullet}, I^{\bullet}) $ is quasi-isomorphic to zero that the same is true of $\fgt (\uhom^{\bullet}(C^{\bullet}, I^{\bullet}))$. This however is absolutely false because an exact sequence of sheaves need not be exact on sections.
\end{remark}

\begin{proposition} \label{prop:appsheafification}
Let $\mathscr{F},\mathscr{G}$ be any two objects of $D(\operatorname{Mod}(\mathscr{O}_X))$. Then \[\mathscr{F} \otimes^{\operatorname{L}} \mathscr{G} = (R\fgt(\mathscr{F}) \otimes^{\operatorname{L}} R\fgt(\mathscr{G}))^\#.\] 
\end{proposition}
\begin{proof}
 Let $\mathscr{H}$ be any object of $D( \operatorname{Mod}(\mathscr{O}_X))$. Then we have
\begin{eqnarray*} \Hom_{\dm}((R\fgt(\mathscr{F}) \otimes^{\operatorname{L}} R\fgt( \mathscr{G}))^\#, \mathscr{H}) &=&   \Hom_{\dpm}(R\fgt(\mathscr{F}) \otimes^{\operatorname{L}} R\fgt( \mathscr{G}),R\fgt( \mathscr{H})       )   \\
 &=& \Hom_{\dpm}(R\fgt(\mathscr{F}), R\underline{\Hom}(R\fgt( \mathscr{G}), R\fgt(\mathscr{H})      ))  \\
 &=&   \Hom_{\dpm}(R\fgt(\mathscr{F}),R\fgt R\underline{\Hom}( \mathscr{G}, \mathscr{H}       ))   \\ 
 &=& H^0( \Gamma(X,  R\uhom(R\fgt(\mathscr{F}),R\fgt R\underline{\Hom}( \mathscr{G}, \mathscr{H}       )) ))  \\ 
 &=& H^0( \Gamma(X, R\fgt  R\uhom(\mathscr{F},R\underline{\Hom}( \mathscr{G}, \mathscr{H}       ))  ))\\
 &=& H^0( R\Gamma(X,  R\uhom(\mathscr{F},R\underline{\Hom}( \mathscr{G}, \mathscr{H}       ))  ))\\
 &=&  \Hom_{\dm}(\mathscr{F},R\underline{\Hom}( \mathscr{G}, \mathscr{H}       ))   \\ 
 &=& \Hom_{\dm}(\mathscr{F}\otimes^{\operatorname{L}} \mathscr{G}, \mathscr{H}       ).  \end{eqnarray*} 
 By Yoneda's Lemma,  $  \mathscr{F}\otimes^{\operatorname{L}} \mathscr{G} = (R\fgt(\mathscr{F}) \otimes^{\operatorname{L}} R\fgt( \mathscr{G}))^\#$. 
\end{proof} 

\begin{corollary}\label{cor:appsheafification}
For any $\mathscr{F},\mathscr{G} \in D(\operatorname{Mod}(\mathscr{O}_X))$,  $\mathscr{F} \otimes^{\operatorname{L}} \mathscr{G}$ is the (derived) sheafification of the (derived) presheaf  
\[ U \mapsto R\Gamma(U,\mathscr{F}) \otimes^{\operatorname{L}}_{\mathscr{O}_X(U)} R\Gamma(U,\mathscr{G}).\] 
\begin{proof}
By Proposition \ref{prop:appsheafification}, it is enough to note that the diagram
 \[ \begin{tikzcd} \dm \ar{r}{R\fgt} \ar[swap]{dr}{R\Gamma}  & \dpm \ar{d}{\Gamma}  \\
{}  & D(\text{Ab}) \end{tikzcd} \] 
is commutative, and to prove the following fact. For any $\mathcal{F}, \mathcal{G} \in D(\operatorname{PMod}(\mathscr{O}_X))$, the derived category of \textit{presheaves},
\[ \Gamma(U, \mathcal{F} \otimes^{\operatorname{L}} \mathcal{G}) = \Gamma(U, \mathcal{F}) \otimes^{\operatorname{L}} \Gamma(U,\mathcal{G}).\]
Let $K^\bullet$ be a $K$-flat complex that is quasi-isomorphic to $\mathcal{F}$, and $G^{\bullet}$ any complex representing $\mathcal{G}$. The complex $\operatorname{Tot}(K^\bullet \otimes G^\bullet)$ represents $\mathcal{F} \otimes^{\operatorname{L}} \mathcal{G}$ in $D(\operatorname{PMod}(\mathscr{O}_X))$. 
Note that the formation of the total complex commutes with taking sections in the category of presheaves.  Therefore,
\begin{eqnarray*} \Gamma(U, \mathcal{F} \otimes^{\operatorname{L}} \mathcal{G})  &=& \Gamma(U, \operatorname{Tot}(K^\bullet \otimes G^\bullet)) \\
&=& \operatorname{Tot}( \Gamma(U, K^\bullet \otimes G^\bullet))\\
&=& \operatorname{Tot}(\Gamma(U, K^\bullet) \otimes \Gamma(U,G^\bullet)) .
\end{eqnarray*}

\noindent To complete the proof, we must show that if $K^\bullet$ is a $K$-flat complex of presheaves, then $\Gamma(U, K^\bullet)$ is a $K$-flat complex of modules. To this end, let $M^\bullet$ be an acyclic complex of modules, and $\underline{M}^\bullet$ the constant complex of presheaves associated to $M^\bullet$. Then
\begin{eqnarray} \operatorname{Tot}(\Gamma(U, K^\bullet) \otimes M^\bullet) &=& \operatorname{Tot}(\Gamma(U, K^\bullet \otimes \underline{M}^\bullet) \\
&=& \Gamma(U, \operatorname{Tot}( K^\bullet \otimes \underline{M}^\bullet)). \label{eq:acycliccomp}
\end{eqnarray}
Since $M^\bullet$ is acyclic, $\underline{M}^\bullet$ is acyclic. By assumption that $K^\bullet$ is $K$-flat,  $\operatorname{Tot}( K^\bullet \otimes \underline{M}^\bullet)$ is acyclic. The functor $\Gamma(U,-)$ is exact (in the category of presheaves!) and thus commutes with cohomology. Therefore (\ref{eq:acycliccomp}) is acyclic which shows that $\Gamma(U, K^\bullet)$ is a $K$-flat complex of modules.
\end{proof}

\end{corollary}

\bibliography{bib} 
\bibliographystyle{amsalpha} 

\end{document}